\documentclass[leqno]{article} %12pt

\usepackage[english]{babel}
\usepackage[T1]{fontenc}

\usepackage{amsmath}
\usepackage{amssymb}
\usepackage{amsfonts}
\usepackage{enumerate}
\usepackage{vmargin}
\usepackage[all]{xy}
\usepackage{mathrsfs}
\usepackage{mathtools}
\usepackage{lmodern}
\usepackage[pagebackref=true]{hyperref}

\setmarginsrb{3cm}{3cm}{3.5cm}{3cm}{0cm}{0cm}{1.5cm}{3cm}
\newcommand{\C}{\ensuremath{\mathbb{C}}}

\newcommand{\E}{\ensuremath{\mathbb{E}}}

\newcommand{\N}{\ensuremath{\mathbb{N}}}
\newcommand{\R}{\ensuremath{\mathbb{R}}}

\newcommand{\Z}{\ensuremath{\mathbb{Z}}}

\renewcommand{\leq}{\ensuremath{\leqslant}}
\renewcommand{\geq}{\ensuremath{\geqslant}}
\newcommand{\qed}{\hfill \vrule height6pt  width6pt depth0pt}

\newcommand{\norm}[1]{\left\Vert#1\right\Vert}
\newcommand{\ot}{\otimes}
\newcommand{\epsi}{\varepsilon}
\newcommand{\ovl}{\overline}

\newcommand{\rad}{\mathrm{rad}}

\newcommand{\supp}{\mathrm{supp}}

\newcommand{\dist}{\mathrm{dist}}
\newcommand{\Id}{\mathrm{Id}}

\newcommand{\HI}{H^\infty}
\newcommand{\Hor}{\mathcal{H}}
\newcommand{\type}{\mathrm{type}\,}
\newcommand{\cotype}{\mathrm{cotype}\,}
\newcommand{\dyad}{\varphi}
\newcommand{\cutoff}{\chi}
\newtheorem{thm}{Theorem}[section]
\newtheorem{defi}[thm]{Definition}

\newtheorem{prop}[thm]{Proposition}

\newtheorem{cor}[thm]{Corollary}
\newtheorem{lemma}[thm]{Lemma}

\newtheorem{remark}[thm]{Remark}

\newtheorem{ass}[thm]{Assumption}

\newenvironment{proof}[1][]{\noindent {\it Proof #1} : }{\hbox{~}\qed
\smallskip
}

\numberwithin{equation}{section}
\allowdisplaybreaks
\begin{document}
\selectlanguage{english}
\title{\bfseries{Maximal H\"ormander Functional Calculus on $L^p$ spaces and UMD lattices}}
\date{June 2021}
\author{\bfseries{Luc Deleaval and Christoph Kriegler}}

\maketitle

%%%%%%%%%%%%%%%%%%%%%%%%%%%%%%%%%%%%%%%%%%%%%%%%%%%%%%%%%%%%%%%%
%%%%%%%%%%%%%%%%%%%%%%%%%%%%%%%%%%%%%%%%%%%%%%%%%%%%%%%%%%%%%%%%
\begin{abstract}
Let $A$ be a generator of an analytic semigroup having a H\"ormander functional calculus on $X = L^p(\Omega,Y)$, where $Y$ is a UMD lattice.
Using methods from Banach space geometry in connection with functional calculus, we show that for H\"ormander spectral multipliers decaying sufficiently fast at $\infty$, there holds a maximal estimate $\| \sup_{t > 0} |m(tA)f|\, \|_{L^p(\Omega,Y)} \lesssim \|f\|_{L^p(\Omega,Y)}$.
We also show square function estimates $\left\| \left( \sum_k \sup_{t > 0} |m_k(tA)f_k|^2 \right)^{\frac12} \right\|_{L^p(\Omega,Y)} \lesssim \left\| \left( \sum_k |f_k|^2 \right)^{\frac12} \right\|_{L^p(\Omega,Y)}$ for suitable families of spectral multipliers $m_k$, which are even new for the euclidean Laplacian on scalar valued $L^p(\R^d)$.
As corollaries, we obtain maximal estimates for wave propagators  and Bochner-Riesz means.
Finally, we illustrate the results by giving several examples of operators $A$ that admit a H\"ormander functional calculus on some $L^p(\Omega,Y)$ and discuss examples of lattices $Y$ and non-self-adjoint operators $A$ fitting our context.
\end{abstract}

%%%%%%%%%%%%%%%%%%%%%%%%%%%%%%%%%%%%%%%%%%%%%%%%%%%%%%%%%%%%%%%%
%%%%%%%%%%%%%%%%%%%%%%%%%%%%%%%%%%%%%%%%%%%%%%%%%%%%%%%%%%%%%%%%

\makeatletter
 \renewcommand{\@makefntext}[1]{#1}
 \makeatother
 \footnotetext{
%  42A45 Multipliers.
%  42B25 Maximal functions, Littlewood-Paley theory.
%  47A60 Functional calculus.
}

 \tableofcontents

\section{Introduction}

Let $m$ be a bounded function on $(0,\infty)$ and $u(m)$ the operator on $L^p(\R^d)$ defined by $[m(-\Delta)g]\hat{\phantom{i}} = [u(m)g]\hat{\phantom{i}}=m(\|\xi\|^2)\hat{g}(\xi).$
H\"ormander's theorem on Fourier multipliers \cite[Theorem 2.5]{Hor} asserts that $u(m) : L^p(\R^d) \to L^p(\R^d)$ is bounded for any $p \in (1,\infty)$ provided that for some integer $\alpha$ strictly larger than $\frac{d}{2},$
\begin{equation}
\label{equ-intro-classical-Hormander}
\|m\|_{\Hor^\alpha_2}^2 := \max_{k=0,1,\ldots,\alpha} \sup_{R > 0} \frac{1}{R}\int_{R}^{2R} \Bigl| t^k \frac{d^k}{dt^k} m(t) \Bigr|^2 \,dt < \infty.
\end{equation}
This theorem has many refinements and generalisations to various similar contexts.
Namely, one can generalise to non-integer $\alpha$ in \eqref{equ-intro-classical-Hormander} to get larger (for smaller $\alpha$) admissible classes $\Hor^\alpha_2 = \{ m : (0,\infty) \to \C \text{ bounded and continuous} :\: \|m\|_{\Hor^\alpha_2} < \infty \}$ of multiplier functions $m$ (see Subsection \ref{subsec-abstract-Hormander}).
Moreover, it has been a deeply studied question over the last years to what extent one can replace the ordinary Laplacian subjacent to H\"ormander's theorem by other operators $A$ acting on some $L^p(\Omega)$ space.
A theorem of H\"ormander type holds true for many elliptic differential operators $A,$ including sub-Laplacians on Lie groups of polynomial growth, Schr\"odinger operators and elliptic operators on Riemannian manifolds, see \cite{ACMM,Alex,ChHa,Christ,Duong,DuOS,Mul,MuSt}.
More recently, spectral multipliers have been studied for operators acting on $L^p(\Omega)$ only for a strict subset of $(1,\infty)$ of exponents $p$ \cite{Bl,CDLWY,CDY,CO,COSY,KuUhl,KU2,SYY} and for abstract operators acting on Banach spaces \cite{KrW3}.
A spectral multiplier theorem means then that the linear and multiplicative mapping
\begin{equation}
\label{equ-intro-homomorphism}
\Hor^\alpha_2 \to B(X), \: m \mapsto m(A),
\end{equation}
e.g. extending the ad hoc functional calculus for rational functions $m$ with poles outside $[0,\infty)$,
is bounded, where typically $X = L^p(\Omega).$
Following the theory of Fourier multipliers on Bochner spaces developed by Hyt\"onen \cite{Hy1,Hy2} and Girardi and Weis \cite{GiWe} (see also more recently Rozendaal and Veraar \cite{RoVe}), one can also ask if the calculus \eqref{equ-intro-homomorphism} can be extended to $X = L^p(\Omega,Y)$ where $Y$ is a UMD lattice.
This is answered positively in \cite{DKK} if $A$ generates a self-adjoint semigroup satisfying (generalised) Gaussian estimates.
This extends a large class of examples hereabove to a vector valued context \cite[Section 5]{DKK}.
The operator $A$ appearing in \cite{DKK} is always a $c_0$-semigroup generator on $X$, a framework that we keep in the present article.
More precisely, we consider $0$-sectorial operators $A$ that are negative generators of analytic semigroups on $\C_+$, thus allowing the machinery of $\HI(\Sigma_\omega)$ calculus for any angle $\omega \in (0,\pi)$.
Although most of the examples of such operators are self-adjoint on $L^2(\Omega)$, there are also non-self-adjoint natural ones in the context of weighted $L^2$ space, see Subsection \ref{subsec-examples-GE}.
For an operator $A$ that has a functional calculus on (a subspace of) a UMD lattice $Y(\Omega')$ but not on $L^p(\Omega')$  for $1 < p \neq 2 < \infty$, see the example in Subsection \ref{subsec-examples-UMD-lattices}.

One of the main results of the present article is to \emph{deduce} from a calculus in \eqref{equ-intro-homomorphism} for an abstract generator $A$ a stronger statement of \emph{maximal} bound of the form
\begin{equation}
\label{equ-intro-maximal-Hor}
\left\| \sup_{t > 0} |m(tA) f| \, \right\|_{L^p(\Omega,Y)} \leq C \|f\|_{L^p(\Omega,Y)}
\end{equation}
for \emph{certain} H\"ormander spectral multipliers (see Corollary \ref{cor-main-C0} and Proposition \ref{prop-main-exp}).
Early results of the type \eqref{equ-intro-maximal-Hor} in the euclidean case for non-special spectral multipliers $m$ are due to Carbery \cite{Crb}, Dappa and Trebels \cite{DaTr}, Seeger \cite{See} and to Rubio de Francia \cite{RdF2}.
It is known already from \cite{CGHS} in the euclidean case, $A$ the Laplacian operator and $Y = \C$ that \eqref{equ-intro-maximal-Hor} cannot hold for all $\Hor^\alpha_2$ H\"ormander multipliers $m$, even for a large prescribed derivation order $\alpha$ in \eqref{equ-intro-classical-Hormander}.
Other assumptions are therefore needed.

In this direction, if $\Omega = G$ is a stratified Lie group and $A$ is a left invariant sublaplacian, Mauceri and Meda proved in \cite[Theorem 2.6]{MauMeda} that \eqref{equ-intro-maximal-Hor} holds provided that
\begin{equation}
\label{equ-intro-summability}
\sum_{n \in \Z} \|m(2^n \cdot) \dyad_0\|_{W^c_2(\R)} < \infty
\end{equation}
where $\dyad_0 \in C^\infty_c(0,\infty)$ satisfies $\dyad_0(t) = 1$ for $t \in (1,2)$, $W^c_2(\R)$ stands for the usual Sobolev space with derivation exponent $c > Q(\frac1p - \frac12) + \frac12$ $(1 < p \leq 2)$ or $c > (Q-1)(\frac12 - \frac1p) + \frac12$ $(2 \leq p \leq \infty)$, $Q$ denoting the homogeneous dimension of the group.
Note that it is well-known that if the sum over $n$ in \eqref{equ-intro-summability} is replaced by the supremum, then one obtains the natural generalisation of condition \eqref{equ-intro-classical-Hormander} for non-integer $c > \frac12$.

A more recent result guaranteeing \eqref{equ-intro-maximal-Hor} for the usual Laplace operator on euclidean space is that of \cite[Theorem 1.2, Corollary 1.3]{CGHS}, where summation \eqref{equ-intro-summability} is relaxed to 
\[ \sum_{n \in \Z} \frac{1}{|n|+1}\|m(2^n\cdot)\dyad_0\|_{W^{\frac{d}{\min(p,2)}}_2(\R)} < \infty\]
(even relaxed to a certain rearrangement $m(2^n\cdot) \leadsto m(2^{k_n} \cdot)$ that minimizes the sum).
In \cite{GHS}, a similar condition to the above is shown to yield a maximal estimate, with a slightly different Sobolev space and where the prefactor is $\frac{1}{(|n|+1)\log(|n|+2)}$.
Moreover, in \cite{Choi}, Choi extends this result again to the Mauceri-Meda setting of left-invariant sublaplacians on stratified Lie groups under a slightly more restrictive summation condition involving a supplementary $\log(|n|+2)$ factor in the numerator.

Our first main result reads as follows.
Here and henceforth we write in short $L^p(Y)$ for the Bochner space $L^p(\Omega,Y)$.
We also use frequently $\Lambda^\beta_{2,2}(\R_+) = \{ f : \R_+ \to \C,\: f \circ \exp \in W^\beta_2(\R) \}$ which is the usual Sobolev space transferred to $\R_+$ via the exponential function, and $\|f\|_{\Lambda^\beta_{2,2}(\R_+)} = \|f\circ \exp\|_{W^\beta_2(\R)}$.

\begin{thm}[see Theorem \ref{thm-main} and Corollary \ref{cor-main-full-support}]
\label{equ-intro-main-theorem}
Let $Y = Y(\Omega')$ be a UMD lattice, $1 < p < \infty$ and $(\Omega,\mu)$ a $\sigma$-finite measure space.
Let $\beta > \frac12$.
Let $A$ be a $0$-sectorial operator on $L^p(Y)$.
Assume that $A$ has a $\Hor^\alpha_2$ calculus on $L^p(Y)$ for some $\alpha > \frac12$.
Let $m \in W^{c}_2(\R)$ be a spectral multiplier with $m(0) = 0$ and $c > \alpha + \max\left(\frac12,\frac{1}{\type L^p(Y)} - \frac{1}{\cotype L^p(Y)}\right) + \frac12 + \beta$ such that
\begin{equation}
\label{equ-1-intro-thm-main}
\sum_{n \in \Z} \|m(2^n \cdot) \dyad_0 \|_{W^c_2(\R)} < \infty
\end{equation}
for some $\dyad_0 \in C^\infty_c(0,\infty)$ with $\dyad_0(t) = 1$ for $t \in (1,2)$.
Then for a.e. $(x,\omega) \in \Omega \times \Omega'$, $t \mapsto m(tA)f(x,\omega)$ belongs to $C_0(\R_+)$ and
\begin{equation}
\label{equ-2-intro-thm-main}
\| \sup_{t > 0} |m(tA)f|\,\|_{L^p(Y)} \lesssim \| t \mapsto m(tA)f\|_{L^p(Y(\Lambda^\beta_{2,2}(\R_+)))} \lesssim \sum_{n \in \Z} \|m(2^n \cdot) \dyad_0 \|_{W^c_2(\R)}\|f\|_{L^p(Y)}.
\end{equation}
\end{thm}

Here as in \cite{MauMeda} we observe that we have a continuous injection $\Lambda^\beta_{2,2}(\R_+) \hookrightarrow C_0(\R_+)$, the latter space being naturally equipped with the supremum norm.
The strategy of the proof of Theorem \ref{equ-intro-main-theorem}, valid for any semigroup generator with H\"ormander calculus, is to exploit the abstract approach of that functional calculus due to the second author and Weis \cite{KrW3}.
It allows in the general context of Theorem \ref{equ-intro-main-theorem} above to expand the norm in $L^p(Y(\Lambda^\beta_{2,2}(\R_+)))$ via the Littlewood-Paley decomposition (see Lemma \ref{lem-Hormander-calculus-Paley-Littlewood}) into pieces of compactly supported spectral multipliers.
This explains already the use of $\Lambda^\beta_{2,2}(\R_+)$ which is a Hilbert space, thus having much nicer geometrical properties than $C_0(\R_+)$, and being more appropriate to our Banach space geometrical proof.
A second part is to use for these compactly supported spectral multipliers a representation formula from Lemma \ref{lem-representation-formula-wave-operators} below which in turn allows to transfer the so-called $R$-boundedness from the spectral multipliers $(1 + 2^n A)^{-\gamma} \exp(i2^n t A)$ to the Littlewood-Paley pieces.
The notion of $R$-boundedness, well-established by now as being of fundamental importance in the context of functional calculus \cite{KW04} is explained in Subsection \ref{subsec-R-boundedness} and allows to reassemble the Littlewood-Paley pieces and thus to conclude \eqref{equ-2-intro-thm-main}.

Note that the quantities $\type L^p(Y) \in (1,2]$ and $\cotype L^p(Y) \in [2,\infty)$ in the hypotheses of Theorem \ref{equ-intro-main-theorem} refer to Rademacher (co-)type explained in Subsection \ref{subsec-R-boundedness} and are again of Banach space geometrical nature suitable for our context.
The term $\frac{1}{\type L^p(Y)} - \frac{1}{\cotype L^p(Y)}$ in Theorem \ref{equ-intro-main-theorem} is then due to the use of $R$-boundedness of $\Hor^\alpha_2$ spectral multipliers in the proof.
We do not expect that Theorem \ref{equ-intro-main-theorem} comes in this form with optimal exponent $c$, however the square function estimates that we obtain in Theorem \ref{thm-intro-main-C0} below with same assumptions as Theorem \ref{equ-intro-main-theorem} necessarily imply type and cotype in a form as in Theorem \ref{equ-intro-main-theorem}.
Indeed in Proposition \ref{prop-type-cotype-necessary}, we show that $R$-bounded $\Hor^\alpha_s$ calculus (= square function estimates) of $-\Delta \ot \Id_{L^r(\Omega')}$ on $L^2(\R,L^r(\Omega'))$ necessarily implies that $\alpha \geq \frac{1}{\type L^r(\Omega')} - \frac{1}{\cotype L^r(\Omega')}$.
Note that optimal exponents of H\"ormander calculus is a difficult question.
Already in the non-maximal form \eqref{equ-intro-homomorphism}, scalar case $Y = \C$ and $A = - \Delta$ the euclidean Laplacian on $L^p(\R^d)$, the optimal parameter $\alpha$ is not known today.
Moreover if that optimal parameter had the expected value $\alpha > \max\left( d \left| \frac1p - \frac12 \right|,\frac12 \right)$, this would solve the famous Bochner-Riesz conjecture.

In view of applications of Theorem \ref{equ-intro-main-theorem} above, note that if $m$ is e.g. continuous on $[0,\infty)$, then \eqref{equ-1-intro-thm-main} forces $m(0)$ to be equal to $0$.
However, $m(0) \neq 0$ should not be the major obstacle for a maximal estimate \eqref{equ-2-intro-thm-main} to hold.
Indeed, it is well-known that if $A$ generates a sub-markovian semigroup, then $\|\sup_{t > 0} |e^{-tA}f| \,\|_{L^p(\Omega)} \lesssim \|f\|_{L^p(\Omega)}$ holds for $1 < p < \infty$, whereas $m(0) = 1$ for $m(\lambda) = e^{-\lambda}$ in this case.
See also \cite[Theorem 2]{Xu2015} for an extension of this result to the case $X = L^p(\Omega,Y)$, $Y$ being a UMD lattice, $A$ of the special form $A = A_0 \otimes \Id_Y$ and $\exp(-tA_0)$ being regular contractive.
We therefore extend for such regular contractive semigroups Theorem \ref{equ-intro-main-theorem} to spectral multipliers sufficiently smooth on $[0,1]$ and satisfying only the one-sided summability $\sum_{n = 0}^\infty \|m(2^n \cdot)\dyad_0\|_{W^c_2(\R)} < \infty$ (see Proposition \ref{prop-main-exp}).
Note that Wr\'obel \cite[Theorem 3.1]{Wro} has obtained a similar result in the scalar case $Y = \C$ with slightly different parameters in $W^c_q(\R)$ of differentiation $c$ and integration $q \neq 2$.

As a corollary, specialising $m(\lambda)$ to the wave spectral multiplier or the Bochner-Riesz spectral multiplier, we obtain then

\begin{cor}[see Corollary \ref{cor-wave-Bochner-Riesz-maximal}]
\label{cor-intro-wave-Bochner-Riesz-maximal}
Assume that $Y$ is a UMD lattice, $1 < p < \infty$ and $(\Omega,\mu)$ a $\sigma$-finite measure space.
Let $A = A_0 \otimes \Id_Y$ be a $0$-sectorial operator on $L^p(Y)$ and assume that $A$ has a $\Hor^\alpha_2$ calculus on $L^p(Y)$ for some $\alpha > \frac12$.
Finally, assume that $\exp(-tA_0)$ is lattice positive and contractive on $L^p(\Omega)$.
Let
\[\delta  > \alpha + \max \left( \frac12, \frac{1}{\type L^p(Y)} - \frac{1}{\cotype L^p(Y)} \right) + 1. \]
Then the wave operators associated with $A$ satisfy the maximal estimate
\[ \left\| \sup_{t > 0} \left| (1  + t A)^{-\delta} \exp(itA) f \right| \, \right\|_{L^p(Y)} \leq C \| f \|_{L^p(Y)} . \]
Moreover, let 
\[ \gamma > \alpha + \max \left( \frac12, \frac{1}{\type L^p(Y)} - \frac{1}{\cotype L^p(Y)} \right) 
+ \frac12 . \]
Then the Bochner-Riesz means associated with $A$ satisfy the maximal estimate
\[ \left\| \sup_{t > 0} \left| \left( 1 - \frac{A}{t} \right)_+^\gamma f \right| \, \right\|_{L^p(Y)} \leq C \| f \|_{L^p(Y)} . \]
\end{cor}

Our approach of the proof of Theorem \ref{equ-intro-main-theorem}, using $R$-boundedness and square functions, has the advantage of flexibility to also include square function estimates in the conclusions.
We thus obtain

\begin{thm}[see Corollaries \ref{cor-main-full-support} and \ref{cor-main-C0}]
\label{thm-intro-main-C0}
Let $Y$ be a UMD lattice, $1 < p < \infty$, $\beta > \frac12$ and $(\Omega,\mu)$ a $\sigma$-finite measure space.
Let $A$ be a $0$-sectorial operator on $L^p(Y)$.
Assume that $A$ has a $\Hor^\alpha_2$ calculus on $L^p(Y)$ for some $\alpha  > \frac12$.
Choose some
\[ c > \alpha + \max \left( \frac12, \frac{1}{\type L^p(Y)} - \frac{1}{\cotype L^p(Y)} \right) + \frac12 + \beta . \]
Let $(m_k)_{k \in \N}$ be a sequence of spectral multipliers with $m_k(0) = 0$ such that $\sum_{n \in \Z} \sup_k \|m_k(2^n \cdot)\dyad_0\|_{W^c_2(\R)} < \infty$.
Then
\begin{align*}
\left\| \left( \sum_k \sup_{t > 0} |m_k(tA)f_k|^2 \right)^{\frac12} \right\|_{L^p(Y)} &  \lesssim \left\| \left( \sum_k \|t \mapsto m_k(tA)f_k\|_{\Lambda^\beta_{2,2}(\R_+)}^2 \right)^{\frac12} \right\|_{L^p(Y)} \\
& \lesssim \sum_{n \in \Z} \sup_k \|m_k(2^n \cdot)\dyad_0\|_{W^c_2(\R)} \left\| \left( \sum_k |f_k|^2 \right)^{\frac12} \right\|_{L^p(Y)} . \end{align*}
\end{thm}

Again, there is a version of Theorem \ref{thm-intro-main-C0} including the possibility $m_k(0) \neq 0$, see Proposition \ref{prop-main-exp}, and as in \cite{Choi}, there is also a certain flexibility of possible rearrangements of the pieces $\|m_k(2^n\cdot)\dyad_0\|_{W^c_2(\R)}$ in $n \in \Z$.

Our ambient space $\Lambda^\beta_{2,2}(\R_+)$ does not only inject into $C_0(\R_+)$, but also into the space $V^q(\R_+)$ of functions with bounded $q$-variation for $2 \leq q < \infty$.
In a forthcoming paper \cite{DKqVar}, we will pursue this observation.
We will prove that the spherical mean $A_tf(x) = \int_{S^{d-1}} f(x - t y) d\sigma(y)$ on $\R^d$ (which is a multiplier of the euclidean Laplacian) has bounded $Y$-valued $q$-variation for $f \in L^p(\R^d,Y)$ if $d$ is sufficiently large depending on the geometry of the UMD lattice $Y$.
Furthermore, we will deduce from H\"ormander calculus of an abstract operator $A$ boundedness of the maximal operator $t \mapsto \max_{t \in [t_0,t_1]} |\exp(itA)f(x,\omega)|$ for $0 < t_0 < t_1 < \infty$ on $L^p(Y)$, which then implies pointwise in $(x,\omega) \in \Omega \times \Omega'$ boundedness of the solution of a Schr\"odinger (typically $A = - \Delta$) and wave (typically $A = \sqrt{-\Delta}$) equation, provided the initial data $f$ is sufficiently smooth (lies in a fractional domain space of $A$).
This will provide an abstract version of Carleson's problem concerning the Schr\"odinger equation on $\R^d$ \cite[Theorem p.24]{Car}.

We end this introduction with an overview of the article.
In Section \ref{sec-preliminaries} we introduce the necessary background on Banach space geometry such as $R$-boundedness, UMD lattices, as well as type and cotype.
Moreover, we introduce the notions of H\"ormander functional calculus and the class $\Lambda^\beta_{2,2}(\R_+)$ that appears in the results above.

Then in Section \ref{sec-main}, we state and prove Theorem \ref{equ-intro-main-theorem}, Corollary \ref{cor-intro-wave-Bochner-Riesz-maximal} and Theorem \ref{thm-intro-main-C0} above together with several variants.
These include a version of Theorem \ref{equ-intro-main-theorem} for spectral multipliers not necessarily vanishing or decaying at $0$ (see Proposition \ref{prop-main-exp}) and pointwise convergence of the spectral multiplier $m(tA)f(x,\omega)$ as $t \to 0$ or $t \to \infty$ (see Corollary \ref{cor-pointwise-convergence}).
This convergence holds pointwise for a.e. $(x,\omega) \in \Omega \times \Omega'$ in $\C$, as well as pointwise for a.e. $x \in \Omega$ in $Y$ and in $L^p(Y)$.

Finally, in Section \ref{sec-examples}, we illustrate all our results by numerous examples, stemming from \cite{DKK}, of ambient spaces $(\Omega,\mu)$ together with differential operators $A$ that possess a H\"ormander calculus on $L^p(\Omega,Y)$ and hence the preceding sections apply to them (Subsection \ref{subsec-examples-GE}).
Moreover, in Subsection \ref{subsec-examples-UMD-lattices}, we illustrate applications of our results to lattices beyond $L^q(\Omega')$.
We consider UMD lattices $Y(\Omega')$ different from $L^q(\Omega')$ spaces together with the so-called Stokes operator which has an $\HI$ calculus on a subspace $Z$ of $Y(\Omega')$ but not on $L^q(\Omega')$ for $1 < q \neq 2 < \infty$.
For those domains $\Omega' \subseteq \R^d$ (conditions on the smoothness of the boundary of $\Omega'$) such that this operator $A$ has a H\"ormander calculus on $Z$, our Theorem \ref{equ-intro-main-theorem} yields then a maximal estimate
\[ \norm{t \mapsto \sup_{t > 0} |m(tA)f| \, }_{Y} \lesssim \norm{m}_{W^c_2[\frac12,2]} \norm{f}_Z \]
on non-$L^q$-spaces, for the appropriate choice of $c$.
Finally we discuss the example of Amann's coagulation-fragmentation equation which naturally involves a differential operator $A = \sum_{|\alpha| \leq m} a_\alpha(x) D^\alpha$ acting on $L^p(\R^n,Y)$, which is not just of the form $A = A_0 \ot \Id_Y$ for some $0$-sectorial operator $A_0$ acting on $L^p(\R^n)$, but has nontrivial components $a_\alpha(x) \in B(Y)$ acting on $Y$.
Amann gave in \cite{Ama} sufficient conditions under which this operator is sectorial on $L^p(\R^n,Y)$ after some shift.
Also in the particular case when the above operator $A$ models a reaction-diffusion equation, Amann's more natural choice of function space on which $A$ is sectorial, is a lattice different from pure $L^q$ (see \eqref{equ-Amanns-lattice}).
We remark however that in Amann's setting it is not known today under which conditions $A$ has a H\"ormander calculus on $L^p(\R^n,Y)$.
Finally, in Subsection \ref{subsec-type-cotype-necessary}, we discuss the necessity of type and cotype of $L^p(Y)$ in H\"ormander calculus of the euclidean Laplacian on Bochner spaces $L^2(\R^d,Y)$.

\section{Preliminaries}
\label{sec-preliminaries}

In this section, we recall the notions on Banach space geometry and functional calculus that we need in this paper.
For the H\"ormander functional calculus, we only need few facts that we will use as an abstract blackbox in the remainder of the article.

\subsection{$R$-boundedness}
\label{subsec-R-boundedness}

\begin{defi}
Let $X,Y$ be Banach spaces.
We recall that a family $\tau \subseteq B(X,Y)$ is called $R$-bounded, if for a sequence $(\epsi_k)_k$ of independent Rademacher random variables, taking the value $1$ and $-1$ with equal probability $\frac12$, a constant $C > 0$, any $n \in \N$, any $x_1,\ldots,x_n \in X$ and any $T_1, \ldots , T_n \in \tau$, we have
\[ \E \left\| \sum_{k = 1}^n \epsi_k T_k x_k \right\|_Y \leq C \E \left\| \sum_{k = 1}^n \epsi_k x_k \right\|_X .\]
In this case, the infimum over all admissible $C$ is denoted by the $R$-bound $R(\tau)$.
\end{defi}

\begin{remark}
\label{rem-R-bdd-Hilbert}
Clearly, $R(\{T\}) = \|T\|_{B(X,Y)}$ if $\tau = \{ T \}$ is a singleton.
In general, we have $R(\tau) \geq \sup_{T \in \tau} \|T\|_{B(X,Y)}$ above.
If $X$ and $Y$ are (isomorphic to) Hilbert spaces, then a family $\tau \subset B(X,Y)$ is $R$-bounded if and only if $\tau$ is bounded.
\end{remark}

\begin{defi}
\label{def-property-alpha-type-cotype}
Let $X$ be a Banach space and $(\epsi_n)_n$ be a sequence of independent Rademacher variables.
\begin{enumerate}
\item 
We say that $X$ has Pisier's property $(\alpha)$ if there are constants $c_1,c_2 > 0$ such that for any array $(x_{n,k})_{n,k = 1}^N$ in $X$,$(\epsi'_k)_k$ a second sequence of independent Rademacher variables independent of $(\epsi_n)_n$, and $(\epsi''_{n,k})_{n,k}$ a doubly indexed sequence of independent Rademacher variables, the following equivalence holds:
\[ c_1 \E \E' \left\| \sum_{k,n = 1}^N \epsi_n \epsi'_k x_{n,k} \right\|_{X} \leq \E'' \left\| \sum_{k,n = 1}^N \epsi''_{n,k} x_{n,k} \right\|_X \leq c_2 \E \E' \left\| \sum_{k,n = 1}^N \epsi_n \epsi'_k x_{n,k} \right\|_{X} . \] 
\item Let $p \in [1,2]$ and $q \in [2,\infty]$.
We say that $X$ has type $p$ if for some constant $c > 0$ and any sequence $(x_n)_{n = 1}^N$ in $X$, we have
\[ \E \left\| \sum_{n = 1}^N \epsi_n x_n \right\|_X \leq c \left( \sum_{n = 1}^N \|x_n\|^p \right)^{\frac1p} . \]
In this case, we write $\type(X) = p$ (not uniquely determined value).
We say that $X$ has cotype $q$ if for some constant $c > 0$ and any sequence $(x_n)_{n = 1}^N$ in $X$, we have
\[\left( \sum_{n = 1}^N \|x_n\|^q \right)^{\frac1q} \leq c  \E \left\| \sum_{n = 1}^N \epsi_n x_n \right\|_X  . \]
In this case, we write $\cotype(X) = q$ (not uniquely determined value).
\end{enumerate}
\end{defi}

\subsection{UMD lattices}
\label{subsec-UMD-lattices}

In this article, UMD lattices, i.e. Banach lattices which enjoy the UMD property, play a prevalent role.
For a general treatment of Banach lattices and their geometric properties, we refer the reader to \cite[Chapter 1]{LTz}.
We recall now definitions and some useful properties.
A Banach space $Y$ is called UMD space if the Hilbert transform 
\[ H : L^p(\R) \to L^p(\R),\: Hf(x) = \lim_{\epsilon \to 0} \int_{|x-y| \geq \epsilon} \frac{1}{x-y} f(y) \,dy \] 
extends to a bounded operator on $L^p(\R,Y),$ for some (equivalently for all) $1 < p < \infty$ \cite[Theorem 5.1]{HvNVW}.
The importance of the UMD property in harmonic analysis was recognized for the first time by Burkholder \cite{Burk1981,Burk1983}, see also his survey \cite{Burk2001}.
He settled a geometric characterization via a convex functional \cite{Burk1981} and together with Bourgain \cite{Bourgain1983}, they showed that the UMD property can be expressed by boundedness of $Y$-valued martingale sequences.
A UMD space is super-reflexive \cite{Al79}, and hence (almost by definition) B-convex.
As a survey for UMD lattices and their properties in connection with results in harmonic analysis, we refer the reader to \cite{RdF}.

A K\"othe function space $Y$ is a Banach lattice consisting of equivalence classes of locally integrable functions on some $\sigma$-finite measure space $(\Omega',\mu')$ with the additional properties
\begin{enumerate}
\item If $f :\: \Omega' \to \C$ is measurable and $g \in Y$ is such that $|f(\omega')| \leq |g(\omega')|$ for almost every $\omega' \in \Omega'$, then $f \in Y$ and $\|f\|_Y \leq \|g\|_Y$.
\item The indicator function $1_A$ is in $Y$ whenever $\mu'(A) < \infty$.
\item Moreover, we will assume that $Y$ has the $\sigma$-Fatou property:
If a sequence $(f_k)_k$ of non-negative functions in $Y$ satisfies $f_k(\omega') \nearrow f(\omega')$ for almost every $\omega' \in \Omega'$ and $\sup_k \|f_k\|_Y < \infty$, then $f \in Y$ and $\|f\|_Y = \lim_k \|f_k\|_Y$.
\end{enumerate}
Note that for example, any $L^p(\Omega')$ space with $1 \leq p \leq \infty$ is such a K\"othe function space.

\begin{lemma}
\label{lem-UMD-lattice-Fatou}
Let $Y$ be a UMD lattice.
Then it has the $\sigma$-Levi property: any increasing and norm-bounded sequence $(x_n)_n$ in $Y$ has a supremum in $Y$.
It also has the Fatou-property and hence the $\sigma$-Fatou property.
%[THE FOLLOWING APPEARS JUST SOME LINES ABOVE, NO NEED TO REPEAT?] any increasing sequence $(x_n)_n$ in $Y$ with a supremum $x \in Y$ satisfies $\|x\|_Y = \sup_n \|x_n\|_Y$.
Note that if $1 < p < \infty$ and $(\Omega,\mu)$ is a $\sigma$-finite measure space, then $L^p(\Omega,Y)$ is again a UMD lattice, so has the above $\sigma$-Levi and $\sigma$-Fatou properties.
\end{lemma}

\begin{proof}
Note that a UMD lattice is reflexive.
Then we refer to \cite[Proposition B.1.8]{Lin}.
\end{proof}

\begin{ass}
In the rest of the paper, $Y = Y(\Omega')$ will always be a UMD space which is also a K\"othe function space, unless otherwise stated.
\end{ass}

\begin{defi}
\label{def-Lambda}
We define
\[ \Lambda^\beta := \Lambda^\beta_{2,2} := \Lambda^\beta_{2,2}(\R_+) := \{ f : \R_+ \to \C : \: f \circ \exp \text{ belongs to }W^\beta_2(\R) \} , \]
where $W^\beta_2(\R)$ denotes the usual Sobolev space defined e.g. via the Fourier transform.
We equip the space with the obvious norm $\|f\|_{\Lambda^\beta_{2,2}} := \|f \circ \exp\|_{W^\beta_2(\R)}$.
The space $\Lambda^\beta_{2,2}(\R_+)$ is a Hilbert space and imbeds into the (non-UMD) lattice $C_0(\R_+)$ for $\beta > \frac12$.
Indeed, this follows from the Sobolev embedding $W^\beta_2(\R) \hookrightarrow C_0(\R)$ for $\beta > \frac12$.
\end{defi}

Let $E$ be any Banach space.
We can consider the vector valued lattice $Y(E) = \{ F : \Omega' \to E :\: F \text{ is strongly measurable and }\omega' \mapsto \|F(\omega')\|_E \in Y\}$ with norm $\|F\|_{Y(E)} = \bigl\| \|F(\cdot)\|_E \bigr\|$.
%If $Y$ and $E$ are Banach lattices that are K\"othe function spaces with a $\sigma$-order continuous norm (e.g. an $L^p$ space for $p < \infty$ \cite[p.~235]{Lin} or more generally, a UMD lattice), then $[Y(E)]' = Y'(E')$ \cite[p.~239, p.~237]{Lin}.
From \cite[Corollary p.~214]{RdF}, we know that if $Y$ is UMD and $E$ is UMD, then also $Y(E)$ is UMD.
Moreover, we shall consider specifically in this article spaces $L^p(\Omega,Y(E))$, with e.g. $E = \Lambda^\beta$ as above.
For the natural identity $L^p(\Omega,Y)(E) = L^p(\Omega,Y(E))$ guaranteed e.g. by reflexivity of $Y$, we refer to \cite[Sections B.2.1, B.2.2, Theorem B.2.7]{Lin}.

\begin{remark}
\label{rem-Lambda-dilation-invariant}
The $\Lambda^\beta$ norm is dilation and inversion invariant, that is, for any $f \in \Lambda^\beta$ and $t > 0$, $\|f(t\cdot)\|_{\Lambda^\beta}  = \|f\|_{\Lambda^\beta}$ and $\left\| f\left(  \frac{1}{(\cdot)} \right) \right\|_{\Lambda^\beta} = \|f\|_{\Lambda^\beta}$.
Suppose $f : \R_+ \to \C$ is measurable and has compact support in $\R_+$.
Then $f$ belongs to $\Lambda^\beta$ iff $f$ belongs to $W^\beta_2(\R)$ and in this case, we have $\|f\|_{\Lambda^\beta} \cong \|f\|_{W^\beta_2(\R)}$, where the equivalence constants depend on the compact support.
%Here $W^\beta_2(\R)$ stands for the usual Sobolev space.
\end{remark}

\begin{proof}
The dilation and inversion invariance of the $\Lambda^\beta$ norm easily follows from translation and negation invariance of the $W^\beta_2$ norm.
For the last part, we refer to \cite[p.~63 (4.13)]{KrPhD}, where it is proved that if $\psi \in C^\infty_c(\R_+)$ and if $(\dyad_n)_{n \in \Z}$ is a dyadic partition of $\R_+$ 
(see Definition \ref{def-dyad} below), then $\sup_{t > 0} \| \dyad_0 f(te^{(\cdot)}) \|_{W^\beta_2(\R)} \cong \sup_{t > 0} \|\psi f(t \cdot)\|_{W^\beta_2(\R)}$.
If $f$ has moreover compact support, then we claim that 
\[ \sup_{t > 0} \| \dyad_0 f(te^{(\cdot)}) \|_{W^\beta_2(\R)} \cong \|f(e^{(\cdot)})\|_{W^\beta_2(\R)} = \|f\|_{\Lambda^\beta} \] and that 
\[ \sup_{t > 0} \|\psi f(t \cdot)\|_{W^\beta_2(\R)} \cong \|f\|_{W^\beta_2(\R)} \]
(with equivalence constants depending on the support of $f$), which then concludes the proof.
On the one hand, the function $f(e^{(\cdot)})$ has compact support in $\R$, so that there exists some $\xi \in C^\infty_c(\R)$ and $t_1,\ldots,t_N \in \R$ such that $\xi \sum_{k = 1}^N \dyad_0(\cdot + t_k) = 1$ on $\supp(f)$.
Then we have
\begin{align*}
\|f(e^{(\cdot)})\|_{W^\beta_2} & = \left\| \xi \sum_{k = 1}^N \dyad_0(\cdot + t_k)  f(e^{(\cdot)})\right\|_{W^\beta_2} \\
& \leq \|\xi\|_{M(W^\beta_2)} \sum_{k = 1}^N \| \dyad_0(\cdot + t_k) f(e^{(\cdot)}) \|_{W^\beta_2} \\
& \lesssim \sum_{k = 1}^N \|\dyad_0 f(e^{-t_k} e^{(\cdot)})\|_{W^\beta_2} \\
& \lesssim \sup_{t > 0} \|\dyad_0 f(te^{(\cdot)}) \|_{W^\beta_2},
\end{align*}
where $M(W^\beta_2)$ stands for the space of bounded pointwise multipliers of $W^\beta_2$, which clearly contains $\xi \in C^\infty_c(\R)$.
In the converse direction, we have for given $t > 0$
\begin{align*}
\| \dyad_0 f(te^{(\cdot)}) \|_{W^\beta_2} & = \|\dyad_0 f(e^{(\cdot + \ln (t))}) \|_{W^\beta_2} \\
& = \| \dyad_0(\cdot - \ln(t)) f(e^{(\cdot)}) \|_{W^\beta_2} \\
& \leq \| \dyad_0(\cdot - \ln(t)) \|_{M(W^\beta_2)} \| f(e^{(\cdot)}) \|_{W^\beta_2} \\
& \lesssim \| f (e^{(\cdot)} ) \|_{W^\beta_2},
\end{align*}
where we have used that the norm of $W^\beta_2(\R)$ and thus of $M(W^\beta_2)$ is translation invariant.
We have shown $\sup_{t > 0} \| \dyad_0 f(te^{(\cdot)}) \|_{W^\beta_2(\R)} \cong \|f(e^{(\cdot)})\|_{W^\beta_2(\R)}$.

On the other hand, by the compact support of $f$, there exists $N \in \N$ such that
\begin{align*}
\|f\|_{W^\beta_2} & = \left\| \sum_{k = -N}^N \dyad_k f \right\|_{W^\beta_2} \leq \sum_{k = -N}^N \| \dyad_k f \|_{W^\beta_2} \\
& \cong_N \sum_{k = -N}^N \|\dyad_0 f(2^k \cdot)\|_{W^\beta_2} \\
& \leq (2N+1) \sup_{t > 0} \|\dyad_0 f(t \cdot)\|_{W^\beta_2}.
\end{align*}
Finally, for the inequality in the converse direction, we note that there is a compact $K \subset \R_+$ such that $\dyad_0 f(t\cdot) = 0$ if $t \not\in K$.
Thus,
\[
\|\dyad_0 f(t\cdot)\|_{W^\beta_2} \leq \| \dyad_0\|_{M(W^\beta_2)} \|f(t\cdot)\|_{W^\beta_2} \lesssim_K \|f\|_{W^\beta_2} .
\]
We have shown $\sup_{t > 0} \|\psi f(t \cdot)\|_{W^\beta_2(\R)} \cong \|f\|_{W^\beta_2(\R)}$ with the choice $\psi = \dyad_0$, and the proof is finished.
\end{proof}

\begin{lemma}
\label{lem-UMD-lattice-Rademacher}
Let $Y = Y(\Omega')$ be a UMD lattice and  $(\epsilon_k)_k$ an i.i.d. Rademacher sequence.
Then we have the norm equivalence
\begin{equation}
\label{equ-Rademacher-square}
\E \biggl\|\sum_{ k = 1}^n \epsilon_k y_k \biggr\|_Y \cong \biggl\|\Bigl( \sum_{k = 1}^n |y_k|^2 \Bigr)^{\frac12}\biggr\|_Y
\end{equation}
uniformly in $n \in \N$.
In particular, this also applies to $L^p(\Omega,Y),\: 1 < p < \infty$.
\end{lemma}

\begin{proof}
As $Y$ is a UMD lattice, it is B-convex.
The result thus follows from \cite{Ma74}.
For the last sentence, we only need to recall that $L^p(\Omega,Y)$ will also be a B-convex Banach lattice.
\end{proof}

In the following, we will make use tacitly of the following Lemma \ref{lem-UMD-lattice-Hilbert-extension}.

\begin{lemma}
\label{lem-UMD-lattice-Hilbert-extension}
\begin{enumerate}
\item 
Let $T : Y \to Z$ be a bounded (linear) operator, where $Y(\Omega')$ and $Z(\Omega'')$ are B-convex Banach lattices.
Then its tensor extension $T \otimes \Id_{\ell^2},$ initially defined on $Y(\Omega') \otimes \ell^2 \subset Y(\Omega',\ell^2)$ is again bounded $Y(\Omega',\ell^2) \to Z(\Omega'',\ell^2).$
In particular, if $Y(\Omega')$ is a UMD lattice, then $Y(\Omega',\ell^2)$ is also a UMD lattice.
\item Let $Y(\Omega')$ be a B-convex Banach lattice and $H$ a Hilbert space.
Then $Y(H)$ has type $p \in (1,2]$ if and only if $Y$ has type $p$.
Also $Y(H)$ has cotype $q \in [2,\infty)$ if and only if $Y$ has cotype $q$.
\end{enumerate}
\end{lemma}

\begin{proof}
1. Let $(e_k)_k$ be the canonical basis of $\ell^2.$
We have 
\begin{align*}
\Bigl\| (T\otimes \Id_{\ell^2})\Bigl(\sum_{k = 1}^n y_k \otimes e_k\Bigr) \Bigr\|_{Z(\Omega'',\ell^2)} & =\biggl\|\Bigl( \sum_{k = 1}^n |T y_k|^2 \Bigr)^{\frac12} \biggr\|_Z \cong \E \Bigl\|\sum_{k = 1}^n \epsilon_k T y_k \Bigr\|_Z \\
& \leq R(\{T\}) \E \Bigl\|\sum_{k =1}^n \epsilon_k y_k \Bigr\|_Y \\
& \cong \|T\| \biggl\|\Bigl( \sum_{k = 1}^n |y_k|^2 \Bigr)^{\frac12} \biggr\|_Y.
\end{align*}
This shows the first part.
For the second part, we note that if $Y(\Omega')$ is UMD, then the Hilbert transform $H : L^p(\R,Y) \to L^p(\R,Y)$ is bounded for all $1 < p < \infty.$
Since $L^p(\R,Y)$ is again a B-convex Banach lattice, by the first part, we have that $H : L^p(\R,Y(\Omega',\ell^2)) \to L^p(\R,Y(\Omega',\ell^2))$ is bounded.
Hence by definition, $Y(\Omega',\ell^2)$ is a UMD (lattice).

2. Going into the Definition \ref{def-property-alpha-type-cotype} and using Kahane's inequality \cite[11.1]{DiJT}, $Y$ has type $p$ iff $T : \ell^p(Y) \to L^p(\Omega,Y),\: (x_n)_n \mapsto (\sum_n \epsi_n x_n)$ is bounded.
Since $\ell^p(Y)$ and $L^p(\Omega,Y)$ are B-convex Banach lattices, we infer by part 1. that $T \otimes \Id_{H} : \ell^p(Y(H)) \to L^p(\Omega,Y(H))$ is bounded, so that $Y(H)$ has type $p$.
For the cotype statement, we argue similarly.
If $Y$ has cotype $q$, then $T : L^q(\Omega,Y) \to \ell^q(Y),\: f \mapsto Pf = (\sum_n \epsi_n (Pf)_n)_n \mapsto (Pf)_n$ is bounded, where $P : L^q(\Omega,Y) \to L^q(\Omega,Y), f \mapsto \sum_n \int_\Omega \epsi_n(x) f(x) dx \epsi_n$ denotes the Rademacher projection, which is bounded since $Y$ is B-convex.
We infer by part 1. that $T \otimes \Id_H : L^q(\Omega,Y(H)) \to \ell^q(Y(H))$ is bounded, so that $Y(H)$ has cotype $q$.
\end{proof}

The following lemma will be used in combination with Proposition \ref{prop-Hormander-calculus-to-R-Hormander-calculus} to follow.

\begin{lemma}
\label{lem-property-alpha}
Let $Y$ be a UMD lattice and $p \in (1,\infty)$.
Then $L^p(\Omega,Y)$ has Pisier's property $(\alpha)$.
\end{lemma}

\begin{proof}
Since $Y$ is UMD, it has finite concavity, and so finite cotype \cite[Proposition 1.f.3]{LTz}.
Thus, also $L^p(\Omega,Y)$ has finite cotype \cite[Theorem 11.12]{DiJT}.
Then according to \cite[N 4.8 - 4.10]{KW04}, the Banach function space $L^p(\Omega,Y)$ has property $(\alpha)$.
\end{proof}

\subsection{Abstract H\"ormander Functional Calculus}
\label{subsec-abstract-Hormander}

We recall the necessary background on functional calculus that we will need in this article.
Let $-A$ be a generator of an analytic semigroup $(T_z)_{z \in \Sigma_\delta}$ on some Banach space $X,$ that is, $\delta \in (0,\frac{\pi}{2}],$ $\Sigma_\delta = \{ z \in \C \backslash \{ 0 \} :\: | \arg z | < \delta \},$ the mapping $z \mapsto T_z$ from $\Sigma_\delta$ to $B(X)$ is analytic, $T_{z+w} = T_z T_w$ for any $z,w \in \Sigma_\delta,$ and $\lim_{z \in \Sigma_{\delta'},\:|z| \to 0 } T_zx = x$ for any $x \in X$ and any strict subsector $\Sigma_{\delta'}$ of $\Sigma_\delta$.
We assume that $(T_z)_{z \in \Sigma_\delta}$ is a bounded analytic semigroup, which means $\sup_{z \in \Sigma_{\delta'}} \|T_z\| < \infty$ for any $\delta' < \delta.$

It is well-known \cite[Theorem 4.6, p. 101]{EN} that this is equivalent to $A$ being $\omega$-sectorial for $\omega = \frac{\pi}{2} - \delta,$ that is,
\begin{enumerate}
\item $A$ is closed and densely defined on $X;$
\item The spectrum $\sigma(A)$ is contained in $\overline{\Sigma_\omega}$ (in $[0,\infty)$ if $\omega = 0$);
\item For any $\omega' > \omega,$ we have $\sup_{\lambda \in \C \backslash \overline{\Sigma_{\omega'}}} \| \lambda (\lambda - A)^{-1} \| < \infty.$
\end{enumerate}
We say that $A$ is strongly $\omega$-sectorial if it is $\omega$-sectorial and has moreover dense range.
If $A$ is $\omega$-sectorial and does not have dense range, but $X$ is reflexive, which will always be the case in this article, then we may take the injective part $A_0$ of $A$ on $\overline{R(A)} \subseteq X$ \cite[Proposition 15.2]{KW04}, which then does have dense range and is strongly $\omega$-sectorial.
Here, $R(A)$ stands for the range of $A.$
Then $-A$ generates an analytic semigroup on $X$ if and only if so does $-A_0$ on $\overline{R(A)}.$
%This parallel will continue this section, i.e. the functional calculus for $A_0$ can be extended to $A$ in an obvious way, see \cite[Illustration 4.87]{KrPhD}.
For $\theta \in (0,\pi),$ let 
\[ \HI(\Sigma_\theta) = \{ f : \Sigma_\theta \to \C :\: f \text{ analytic and bounded} \} \] equipped with the uniform norm $\|f\|_{\infty,\theta}.$
Let further 
\[ \HI_0(\Sigma_\theta) = \bigl\{ f \in \HI(\Sigma_\theta):\: \exists \: C ,\epsilon > 0 :\: |f(z)| \leq C \min(|z|^\epsilon,|z|^{-\epsilon}) \bigr\}.\]
For an $\omega$-sectorial operator $A$ and $\theta \in (\omega,\pi),$ one can define a functional calculus $\HI_0(\Sigma_\theta) \to B(X),\: f \mapsto f(A)$ extending the ad hoc rational calculus, by using a Cauchy integral formula.
If moreover, there exists a constant $C < \infty$ such that $\|f(A)\| \leq C \| f \|_{\infty,\theta},$ then $A$ is said to have a bounded $\HI(\Sigma_\theta)$ calculus and if $A$ has dense range, the above functional calculus can be extended to a bounded Banach algebra homomorphism $\HI(\Sigma_\theta) \to B(X).$
If $A$ has a bounded $\HI(\Sigma_\theta)$ calculus, and does not have dense range, but $X$ is reflexive, then for $f \in \HI(\Sigma_\theta)$ such that $f(0)$ is well-defined, we can define
\[ f(A) = \begin{bmatrix} f(A_0) & 0 \\ 0 & f(0) P_{N(A)} \end{bmatrix} : \: \overline{R(A)} \oplus N(A) \to \overline{R(A)} \oplus N(A) ,\]
where $P_{N(A)}$ denotes the projection onto the null-space of $A$ along the decomposition $X = \overline{R(A)} \oplus N(A)$.
This calculus also has the property $f_z(A) = T_z$ for $f_z(\lambda) = \exp(-z \lambda),\: z \in \Sigma_{\frac{\pi}{2} - \theta}.$
For further information on the $\HI$ calculus, we refer e.g. to \cite{KW04}. We now turn to H\"ormander function classes and their calculi.
\begin{defi}
\label{defi-Hoermander-class}
Let $\alpha > \frac12.$
We define the H\"ormander class by 
\[\Hor^\alpha_2 = \bigl\{ f : [0,\infty) \to \C \text{ is bounded and continuous on }(0,\infty), \:  \underbrace{|f(0)| + \sup_{R > 0} \| \phi f(R \,\cdot) \|_{W^\alpha_2(\R)}}_{=:\|f\|_{\Hor^\alpha_2}}< \infty \bigr\}.\]
Here $\phi$ is any $C^\infty_c(0,\infty)$ function different from the constant 0 function (different choices of functions $\phi$ resulting in equivalent norms) and $W^\alpha_2(\R)$ is the classical Sobolev space.
\end{defi}

The term $|f(0)|$ is not needed in the functional calculus applications of $\Hor^\alpha_2$ if $A$ is in addition injective.
We can base a H\"ormander functional calculus on the $\HI$ calculus by the following procedure.

\begin{defi}
\label{def-Hormander-calculus}
We say that a $0$-sectorial operator $A$ has a bounded $\Hor^\alpha_2$ calculus if for some $\theta \in (0,\pi)$ and any $f \in \HI(\Sigma_\theta),$
$\|f(A)\| \leq C \|f\|_{\Hor^\alpha_2} \left( \leq C' \left(\|f\|_{\infty,\theta} + |f(0)|\right) \right).$
In this case, the $\HI(\Sigma_\theta)$ calculus can be extended to a bounded Banach algebra homomorphism $\Hor^\alpha_2 \to B(X)$ \cite{KrW3}.
We say that $A$ has an $R$-bounded $\Hor^\alpha_2$ calculus, if it has a bounded $\Hor^\alpha_2$ calculus and
$\left\{ m(A) : \: \|m\|_{\Hor^\alpha_2} \leq 1 \right\}$ is $R$-bounded.
\end{defi}

The H\"ormander norm is dilation invariant, i.e. $\|f(t \cdot)\|_{\Hor^\alpha_2} = \|f\|_{\Hor^\alpha_2}$ for any $t > 0$.
Therefore, the following family of (discrete) dilates of a $C^\infty_c(\R_+)$ function will play an important role.

\begin{defi}
\label{def-dyad}
Let $\dyad_0 \in C^\infty_c(\R_+)$ such that $\supp (\dyad_0) \subseteq [\frac12,2]$.
We define for $n \in \Z$ the dilates $\dyad_n(t) = \dyad_0(2^{-n}t)$ so that $\supp (\dyad_n) \subseteq [\frac12 \cdot 2^n , 2 \cdot 2^n]$.
Assume that $\sum_{n \in \Z} \dyad_n(t) = 1$ for any $t > 0$.
Then we call $(\dyad_n)_{n \in \Z}$ a dyadic partition of $\R_+ = (0,\infty)$.
For the existence of such a dyadic partition, we refer to \cite[6.1.7 Lemma]{BeL}.
\end{defi}

In the course of the Maximal H\"ormander Functional Calculus Theorem in Section \ref{sec-main}, we need to decompose general spectral multipliers by means of special spectral multiplier pieces involving the above dyadic partition.
To reassemble the pieces together, mere boundedness of the pieces is not sufficient, and we will need the following self-improvement of a H\"ormander functional calculus.

\begin{prop}
\label{prop-Hormander-calculus-to-R-Hormander-calculus}
Let $A$ be a $0$-sectorial operator on a Banach space $X$ with property $(\alpha)$.
If $A$ has a bounded $\Hor^\alpha_2$ calculus, then it has an $R$-bounded $\Hor^\gamma_2$ calculus
for any parameter $\gamma > \alpha + \frac{1}{\type X} - \frac{1}{\cotype X}$ such that $\gamma \geq \alpha + \frac12$.
\end{prop}

\begin{proof}
This follows from \cite[Lemma 3.9 (3), Theorem 6.1 (2)]{KrW3}, noting that the $\Hor^\beta_r$ class there with $\frac1r > \frac{1}{\type X} - \frac{1}{\cotype X},\:r \in (1,2]$, is larger than our $\Hor^\gamma_2$ class for $\gamma = \beta$.
\end{proof}

\begin{lemma}
\label{lem-R-bounded-wave-operators}
Assume that a $0$-sectorial operator $A$ has an $R$-bounded $\Hor^\alpha_2$ calculus.
Then the following operator family is $R$-bounded with $R$-bound independent of $s \in \R$:
\[ \left\{ \langle s \rangle^{-\alpha} (1 + 2^n A)^{-\alpha} \exp(i2^n s A) : \: n \in \Z \right\} ,\]
where from now on, we use the notation $\langle s \rangle = 1 + |s|$.
\end{lemma}

\begin{proof}
Writing $m_n(\lambda) = \langle s \rangle^{-\alpha} ( 1 + 2^n \lambda)^{-\alpha} \exp(i2^n s\lambda)$, the lemma follows from Definition \ref{def-Hormander-calculus} if we can estimate $\sup_{n \in \Z}\|m_n\|_{\Hor^\alpha_2} < \infty$.
Decompose $m_n(\lambda) = \langle s \rangle^{-\alpha} \frac{(1 + 2^n |s| \lambda)^\alpha}{(1 + 2^n \lambda)^\alpha} \cdot f_n(\lambda)$ with $f_n(\lambda) = ( 1 + 2^n |s| \lambda)^{-\alpha} \exp(i2^n s\lambda)$.
According to \cite[Lemma 3.9 (2)]{KrW3}, $\|f_n\|_{\Hor^\alpha_2} \lesssim 1$.
For the other factor, we have for any $\theta \in (0, \pi)$
\[ \left\| \langle s \rangle^{-\alpha} \frac{(1+ 2^n |s| \lambda)^\alpha}{(1 + 2^n \lambda)^\alpha} \right\|_{H^\infty(\Sigma_{\theta})} = \left\| \langle s \rangle^{-\alpha} \frac{(1+ |s| \lambda)^\alpha}{(1 + \lambda)^\alpha} \right\|_{H^\infty(\Sigma_{\theta})} \lesssim 1, \]
since $|1 + |s| \lambda| \leq (1 + |s|)(1 + |\lambda|) \lesssim_\theta (1 + |s|) \left| 1 + \lambda \right|$.
Now we conclude by the two facts that $\HI(\Sigma_\theta) \hookrightarrow \Hor^\alpha_2$ \cite[Lemma 3.2 (2)]{KrW3} and that $\Hor^\alpha_2$ is a Banach algebra \cite[Lemma 3.2 (1)]{KrW3}.
\end{proof}

The following lemmata concerning decomposition/expansion of spectral multipliers will be used in the proof of Theorem \ref{thm-main}.
Here, Lemma \ref{lem-Hormander-calculus-Paley-Littlewood} is sometimes called Paley-Littlewood equivalence.

\begin{lemma}
\label{lem-Hormander-convergence-lemma}
Let $A$ be a $0$-sectorial operator with $\Hor^\alpha_2$ calculus.
Let $(\dyad_n)_{n \in \Z}$ be a dyadic partition of $\R_+$.
Then for any $x \in \overline{R(A)}$ (e.g. $x = m(A)y$ for some $y \in X$ and $m \in \Hor^\alpha_2$ with $m(0) = 0$), we have $x = \sum_{n \in \Z} \dyad_n(A) x$ (convergence in $X$).
\end{lemma}

In the setting of the above Lemma \ref{lem-Hormander-convergence-lemma}, we obtain that $D_A := \{ \phi(A)x :\: x \in X,\: \phi \in C^\infty_c(\R_+) \}$ is a dense subspace of $\overline{R(A)}$.
In \cite{KrW3}, $D_A$ is called the calculus core of $A$.

\begin{lemma}
\label{lem-representation-formula-wave-operators}
Let $A$ be a $0$-sectorial operator having a $\Hor^\alpha_2$ calculus.
Let $m \in W^\alpha_2(\R)$ with compact support in $\R_+$.
Then for any $x$ belonging to the calculus core $D_A$, we have
\[ m(A)x = \frac{1}{2\pi} \int_\R \hat{m}(s) \exp(isA) x ds ,\]
where the integral is a Bochner integral in $X$.
\end{lemma}

\begin{proof}
This follows from \cite[Proof of Lemma 4.6 (3)]{KrW3}.
\end{proof}

\begin{lemma}
\label{lem-Hormander-calculus-Paley-Littlewood}
Let $A$ be a $0$-sectorial operator with $\Hor^\alpha_2$ calculus for some $\alpha > \frac12$.
Let $(\dyad_n)_{n \in \Z}$ be a dyadic partition of $\R_+$.
Then we have the following so-called Paley-Littlewood decomposition for $x \in \overline{R(A)}$:
\[ \|x\|_X \cong \E \left\| \sum_{n \in \Z} \epsi_n \dyad_n(A) x \right\|_X ,\]
where the series $\sum_{n \in \Z} \dyad_n(A)x$ converges unconditionally in $X$.
\end{lemma}

\begin{proof}
See \cite[Theorem 4.1]{KrW2}, together with the fact that the restriction of $A$ to $\overline{R(A)}$ is a strongly $0$-sectorial operator having a $\Hor^\alpha_2$ calculus, hence a $\mathcal{M}^\beta$ calculus \cite[Proposition 4.9]{KrPhD} needed in this reference.
\end{proof}

\section{The Maximal H\"ormander Functional Calculus Theorem}
\label{sec-main}

In this section, we state and prove the main results Theorems \ref{equ-intro-main-theorem} and \ref{thm-intro-main-C0}, and Corollary \ref{cor-intro-wave-Bochner-Riesz-maximal} from the introduction.
We start with the basic version in Theorem \ref{thm-main} below on spectral multipliers with compact support, and enhance it in several steps (see Corollaries \ref{cor-main-full-support} and \ref{cor-main-C0}, and Proposition \ref{prop-main-exp}) to more general classes of spectral multipliers, to be able to apply it in the important cases of wave operators and Bochner-Riesz means in Corollary \ref{cor-wave-Bochner-Riesz-maximal}.

\begin{thm}
\label{thm-main}
Let $Y$ be a UMD lattice, $1 < p < \infty$ and $(\Omega,\mu)$ a $\sigma$-finite measure space.
Let $\beta \geq 0$.
Let $A$ be a $0$-sectorial operator on $L^p(Y)$.
Assume that $A$ has a $\Hor^\alpha_2$ calculus on $L^p(Y)$ for some $\alpha > \frac12$.
Let $m \in W^{c}_2(\R)$ be a spectral multiplier with $\supp (m) \subseteq [\frac12,2]$, with
\[ c > \alpha + \max\left(\frac12,\frac{1}{\type L^p(Y)} - \frac{1}{\cotype L^p(Y)}\right) + \frac12 +
 \beta. \]
Then
\begin{equation}
\label{equ-1-thm-main}
\| t \mapsto m(tA)f\|_{L^p(Y(\Lambda^\beta_{2,2}(\R_+)))} \leq C \|m\|_{W^{c}_2(\R)} \|f\|_{L^p(Y)}.
\end{equation}
Moreover, let $(m_k)_{k \in \N}$ be a family of spectral multipliers in $W^c_2(\R)$ with $\supp(m_k) \subseteq [\frac12,2]$.
Then
\begin{equation}
\label{equ-2-thm-main}
\left\| \left( \sum_{k \in \N} \| t \mapsto m_k(tA) f_k \|_{\Lambda^\beta_{2,2}(\R_+)}^2 \right)^{\frac12} \right\|_{L^p(Y)} \leq C \sup_{k \in \N} \|m_k\|_{W^c_2(\R)} \left\| \left( \sum_{k \in \N} |f_k|^2 \right)^{\frac12} \right\|_{L^p(Y)} .
\end{equation}
\end{thm}

\begin{proof}
We start with proving \eqref{equ-1-thm-main} and indicate the changements to prove \eqref{equ-2-thm-main} at the end.
Since $\Lambda^\beta = \Lambda^\beta_{2,2}(\R_+)$ is a Hilbert space, for any $n \in \Hor^\alpha_2$, the operator $n(A) \otimes \Id_{\Lambda^\beta}$ extends to a bounded operator on $L^p(Y(\Lambda^\beta))$ according to Lemma \ref{lem-UMD-lattice-Hilbert-extension}.
Thus, $A$ has a bounded $\Hor^\alpha_2$ calculus on $L^p(Y(\Lambda^\beta))$ and thus, satisfies the Paley-Littlewood decomposition from Lemma \ref{lem-Hormander-calculus-Paley-Littlewood} on $L^p(Y(\Lambda^\beta))$.
According to Proposition \ref{prop-Hormander-calculus-to-R-Hormander-calculus}, $A$ has then an $R$-bounded $\Hor^\gamma_2$ calculus on $L^p(Y(\Lambda^\beta))$, where $\gamma > \alpha + \frac12$ and
\[ \gamma > \alpha + \frac{1}{\type L^p(Y(\Lambda^\beta))} - \frac{1}{\cotype L^p(Y(\Lambda^\beta))} = \alpha + \frac{1}{\type L^p(Y)} - \frac{1}{\cotype L^p(Y)},\]
where in the second equality, we have used Lemma \ref{lem-UMD-lattice-Hilbert-extension}.
Also $A$ has an $R$-bounded $\Hor^\gamma_2$ calculus on $L^p(Y)$ and thus, according to Lemma \ref{lem-R-bounded-wave-operators}, the following operator family in $B(L^p(Y))$ 
\begin{equation}
\label{equ-1-proof-thm-main}
\left\{ \langle s \rangle^{-\gamma} (1 + 2^n A)^{-\gamma} \exp(i2^n s A) :\: n \in \Z \right\}
\end{equation}
is $R$-bounded with an $R$-bound independent of $s \in \R$.
We express for $f \in D_A \subseteq L^p(Y)$, where $D_A$ stands for the calculus core from Subsection \ref{subsec-abstract-Hormander}, and $(\dyad_n)_{n \in \Z}$ a dyadic partition of $\R_+$,
\begin{align*}
\MoveEqLeft
\| t \mapsto m(tA) f \|_{L^p(Y(\Lambda^\beta))} \cong \left\| \left( \sum_{n \in \Z} \| t \mapsto \dyad_n(A) m(tA) f \|_{\Lambda^\beta}^2 \right)^{\frac12} \right\|_{L^p(Y)} \\
& = \left\| \left( \sum_{n \in \Z} \| t \mapsto \dyad_n(A) m(2^{-n}tA) f\|_{\Lambda^\beta}^2 \right)^{\frac12} \right\|_{L^p(Y)} \\
& = \left\| \left( \sum_{n \in \Z} \| t \mapsto \psi(t) m(2^{-n}tA) \dyad_n(A) f \|_{\Lambda^\beta}^2 \right)^{\frac12} \right\|_{L^p(Y)} \\
& \lesssim R \left( \left\{ t \mapsto \psi(t) m(2^n t A) : \: n \in \Z \right\}_{L^p(Y) \to L^p(Y(\Lambda^\beta))} \right) \left\| \left( \sum_{n \in \Z} | \dyad_n(A) f |^2 \right)^{\frac12} \right\|_{L^p(Y)}\\
& \cong  R \left( \left\{ t \mapsto \psi(t) m(2^n t A) : \: n \in \Z \right\}_{L^p(Y) \to L^p(Y(\Lambda^\beta))} \right) \left\| f \right\|_{L^p(Y)}.
\end{align*}
Here, we have used the Paley-Littlewood decomposition in the space $L^p(Y(\Lambda^\beta))$ from Lemma \ref{lem-Hormander-calculus-Paley-Littlewood} in the first line and
the dilation invariance of the $\Lambda^\beta$ norm according to Remark \ref{rem-Lambda-dilation-invariant} in the second line.
Moreover, we have used that $\dyad_n(A)m(2^{-n}tA) f = 0$ for $t \not\in [2^{-2},2^2]$ and thus introduced a function $\psi \in C^\infty_c(\R_+)$ with support in $[2^{-3},2^3]$ and $\psi(t) = 1$ for $t \in [2^{-2},2^2]$ in the third line.
Finally, in the fourth line, we used $R$-boundedness together with the square function equivalence to Rademacher sums from Lemma \ref{lem-UMD-lattice-Rademacher} and in the fifth line, the Paley-Littlewood decomposition in the space $L^p(Y)$ from Lemma \ref{lem-Hormander-calculus-Paley-Littlewood}.

It remains to estimate the $R$-bound of $\left\{ t \mapsto \psi(t) m(2^n t A) : \: n \in \Z \right\}_{L^p(Y) \to L^p(Y(\Lambda^\beta))}$.
To this end, we write with Lemma \ref{lem-representation-formula-wave-operators}
\begin{align}
\psi(t) m(2^n t A) f & = \psi(t) m(2^n t A) \phi(2^n A) f  \nonumber \\
& = \frac{1}{2\pi} \int_\R \frac{1}{t} \hat{m}\left( \frac{s}{t} \right) \psi(t) (1 + 2^n A)^{-\gamma} \exp(i2^n s A) (1+2^n A)^\gamma \phi(2^n A) f ds , \label{equ-2-proof-thm-main}
\end{align}
where $\phi \in C^\infty_c(\R_+)$ with $\phi(s) = 1$ for $s \in \supp(m(t \cdot)) \subseteq [2^{-5},2^5]$ where $t \in \supp(\psi)$, so that $m(ts) = m(ts) \phi(s)$ for such $t$.
As $A$ has an $R$-bounded $\Hor^\gamma_2$ calculus, and $\| (1 + (\cdot))^\gamma \phi \|_{\Hor^\gamma_2} < \infty$, the set
\[ \left\{ (1 + 2^n A)^\gamma \phi(2^n A) : \: n \in \Z \right\}_{L^p(Y) \to L^p(Y)} \]
is $R$-bounded.
It remains to estimate the $R$-bound of the following family from $B(L^p(Y),L^p(Y(\Lambda^\beta)))$:
\begin{equation}
\label{equ-3-proof-thm-main}
\left\{ \frac{1}{2\pi} \int_\R \frac1t \psi(t) \hat{m}\left( \frac{s}{t} \right) \langle s \rangle^{\gamma + \delta} \langle s \rangle^{-(\gamma + \delta)} ( 1 + 2^n A)^{-\gamma} \exp(i2^n sA) ds :\: n \in \Z \right\} ,
\end{equation}
where we pick any $\delta > 1$.
We write $h_s(t) = \frac{1}{t} \hat{m} \left( \frac{s}{t} \right) \psi(t) \langle s \rangle^{\gamma + \delta}$ and $T_s^n = \frac{1}{2 \pi} \langle s \rangle^{-(\gamma + \delta)} (1 + 2^n A)^{-\gamma} \exp(i 2^n s A)$.
Let $f_1,\ldots,f_N \in L^p(Y)$.
We have
\begin{align*}
\E \left\| \sum_{k = 1}^N \epsilon_k \int_\R h_s(t) T_s^{n_k} ds f_k \right\|_{L^p(Y(\Lambda^\beta))} & \cong \left\| \left( \sum_k \left\| \int_\R h_s(t) T_s^{n_k} ds f_k \right\|_{\Lambda^\beta}^2 \right)^{\frac12} \right\|_{L^p(Y)} \\
& = \left\| \left( \sum_k \left\| \int_\R h_s(t) T_s^{n_k}(f_k) ds \right\|_{\Lambda^\beta}^2 \right)^{\frac12} \right\|_{L^p(Y)} \\
& \leq \left\| \left( \sum_k \left( \int_\R \|t \mapsto h_s(t)\|_{\Lambda^\beta} | T_s^{n_k}(f_k)| ds \right)^2 \right)^{\frac12} \right\|_{L^p(Y)} \\
& \leq \sup_{s \in \R} \|h_s \|_{\Lambda^\beta} \left\| \left( \sum_k \left( \int_\R |T_s^{n_k}f_k| ds \right)^2 \right)^{\frac12} \right\|_{L^p(Y)} \\
& \leq \sup_{s \in \R} \|h_s\|_{\Lambda^\beta} \int_\R \left\| \left( \sum_k |T_s^{n_k}f_k|^2  \right)^{\frac12} \right\|_{L^p(Y)} ds \\
& \leq \sup_{s \in \R} \|h_s\|_{\Lambda^\beta} \int_\R R\left(\{ T_s^n :\: n \in \Z \}\right) ds  \left\| \left( \sum_k |f_k|^2 \right)^{\frac12} \right\|_{L^p(Y)} \\
& \cong \sup_{s \in \R} \|h_s\|_{\Lambda^\beta} \int_\R R\left(\{ T_s^n :\: n \in \Z \}\right) ds \: \E \left\| \sum_k \epsilon_k f_k \right\|_{L^p(Y)}
\end{align*}
According to \eqref{equ-1-proof-thm-main} and since $\delta > 1$, we have $\int_\R R\left(\{ T_s^n :\: n \in \Z \}\right) ds < \infty$.
It thus only remains to estimate $\sup_{s \in \R} \| h_s \|_{\Lambda^\beta}$.
According to Remark \ref{rem-Lambda-dilation-invariant}, we have with $\xi(t) = t \psi\left( \frac1t \right) \in C^\infty_c(\R_+)$
\[ \| h_s \|_{\Lambda^\beta} = \| \frac{1}{t} \psi(t) \hat{m}\left(\frac{s}{t}\right) \|_{\Lambda^\beta} \langle s \rangle^{\gamma + \delta} = \| \xi(t) \hat{m}(st) \|_{\Lambda^\beta_{2,2}(\R_+)} \langle s \rangle^{\gamma + \delta} \cong \| \xi(t) \hat{m}(st) \|_{W^\beta_2(\R)} \langle s \rangle^{\gamma + \delta}. \]
We start by estimating the $L^2(\R)$ norm.
For $|s| \leq 1$, we have
\[ \|\xi(t) \hat{m}(st)\|_{L^2(\R)}^2 \langle s \rangle^{2(\gamma + \delta)} \lesssim \int_{1/8}^8 |\hat{m}(st)|^2 dt \lesssim \|\hat{m}\|_{L^\infty(\R)}^2 \lesssim \|m\|_{L^1(\R)}^2 \lesssim \|m\|_{L^2(\R)}^2, \]
where we have used that $m$ has compact support in $[\frac12,2]$ in the last estimate.
For $|s| \geq 1$, we have
\begin{align*}
\| \xi(t) \hat{m}(st)\|_{L^2(\R)}^2 \langle s \rangle^{2(\gamma + \delta)} & \lesssim \int_{1/8}^8 |\hat{m}(st)|^2 s^{2(\gamma + \delta - \frac12)} |s| dt \cong \int_{1/8}^8 \left| \hat{m}(st) (st)^{\gamma + \delta - \frac12} \right|^2 |s| dt \\
& = \int_{1/8 s}^{8 s} \left|\hat{m}(t) t^{\gamma + \delta - \frac12}\right|^2 dt \lesssim \|m\|_{W^{\gamma + \delta - \frac12}_2(\R)}^2 \lesssim \|m\|_{W^{\gamma + \delta - \frac12 + \beta}_2(\R)}^2 .
\end{align*}
By a density argument, we can assume that $m \in C^\infty_c(\R)$.
Then it suffices to estimate the $L^2(\R)$ norm of the $\beta$-derivative (defined via its Fourier multiplier).
For $|s| \leq 1$, we have
\begin{align*}
\|\partial^\beta (\xi(t) \hat{m}(st)) \|_{L^2(\R)}^2 \langle s \rangle^{2(\gamma + \delta)} & \lesssim \int_{\frac18}^8 |\partial^\beta \left(\hat{m}(st)\right)|^2 dt = \int_{\frac18}^8 |\partial^\beta(\hat{m})(st)s^\beta|^2 dt \\
& \lesssim \|\partial^\beta(\hat{m})\|_{L^\infty(\R)}^2 \lesssim \|(\cdot)^\beta m\|_{L^1(\R)}^2 \\
& \lesssim \|m\|_{L^2(\R)}^2,
\end{align*}
where we have used that $m$ has compact support in $[\frac12,2]$ in the last estimate.
Finally, for $|s| \geq 1$, we have
\begin{align*}
\MoveEqLeft
\| \partial^\beta(\xi(t) \hat{m}(st)) \|_{L^2(\R)}^2 \langle s \rangle^{2(\gamma + \delta)} \lesssim \int_{\frac18}^8 \left| \partial^\beta\left(\hat{m}(st)\right) s^{\gamma + \delta - \frac12}\right|^2 |s| dt \\
& \cong \int_{1/8}^8 \left| \partial^\beta (\hat{m})(st) (st)^{\gamma+\delta - \frac12 +\beta} \right|^2 |s| dt
= \int_{1/8s}^{8s} \left| \partial^\beta (\hat{m})(t) t^{\gamma + \delta - \frac12 + \beta} \right|^2 dt \\
& \lesssim \| (\cdot)^\beta m\|_{W^{\gamma+\delta - \frac12 +\beta}_2(\R)}^2 \lesssim \|m\|_{W^{\gamma+\delta-\frac12+\beta}_2(\R)}^2.
\end{align*}
Resuming the four estimates above, we deduce
\[ \sup_{s \in \R} \|h_s\|_{\Lambda^\beta} \lesssim \|m\|_{W^{\gamma+\delta-\frac12+\beta}_2(\R)} , \]
and \eqref{equ-1-thm-main} is proved.

We indicate now how to show the square function estimate \eqref{equ-2-thm-main}.
Since $L^p(Y)$ has the $\sigma$-Levi and $\sigma$-Fatou properties according to Lemma \ref{lem-UMD-lattice-Fatou}, we can restrict to a finite family of spectral multipliers $\{m_1,\ldots,m_N\}$.
According to Lemma \ref{lem-property-alpha}, $L^p(Y(\Lambda^\beta))$ has property $(\alpha)$.
Thus, we have with $\epsilon_n'$ and $\epsilon_k$ two independent sequences of Rademachers,
\begin{align*}
\left\| \left( \sum_k \|t \mapsto m_k(tA) f_k \|_{\Lambda^\beta}^2 \right)^{\frac12} \right\|_{L^p(Y)} & \cong \E \left\| \sum_k \epsilon_k (t \mapsto m_k(tA) f_k) \right\|_{L^p(Y(\Lambda^\beta))} \\
& \cong \E \E' \left\| \sum_{n \in \Z} \sum_k \epsilon_k \epsilon_n' (t \mapsto m_k(tA) \dyad_n(A) f_k) \right\|_{L^p(Y(\Lambda^\beta))} \\
& \cong \left\| \left( \sum_{n \in \Z} \sum_k \|t \mapsto m_k(tA) \dyad_n(A) f_k \|_{\Lambda^\beta}^2 \right)^{\frac12} \right\|_{L^p(Y)},
\end{align*}
where we have used the Paley-Littlewood equivalence from Lemma \ref{lem-Hormander-calculus-Paley-Littlewood} on $L^p(Y(\Lambda^\beta))$ in the second equivalence, for fixed choices of signs $\epsilon_k$.
Then, similarly to the proof of \eqref{equ-1-thm-main}, we have
\begin{align*}
\left\| \left( \sum_k \| t \mapsto m_k(tA) f_k\|_{\Lambda^\beta}^2 \right)^{\frac12} \right\|_{L^p(Y)} & \leq R\left(\left\{ t \mapsto \psi(t)m_k(2^ntA) : \: k,n \right\} \right) \left\| \left( \sum_{n \in \Z} \sum_k | \dyad_n(A) f_k|^2 \right)^{\frac12} \right\|_{L^p(Y)} \\
& \cong R\left(\left\{ t \mapsto \psi(t)m_k(2^ntA) : \: k,n \right\} \right) \left\| \left( \sum_k | f_k |^2 \right)^{\frac12} \right\|_{L^p(Y)},
\end{align*}
where we have used the Paley-Littlewood equivalence from Lemma \ref{lem-Hormander-calculus-Paley-Littlewood} on the space $L^p(Y(\ell^2_N))$.
Here, $\psi$ is the same function as in the proof of \eqref{equ-1-thm-main}.
Thus we are reduced to show the $R$-boundedness of the family
\[ \left\{ t \mapsto \psi(t) m_k(2^{n}tA) : \: k \in \{1,\ldots,N\}, n \in \Z \right\}_{L^p(Y) \to L^p(Y(\Lambda^\beta))} . \]
Again similarly to the proof of \eqref{equ-1-thm-main}, we put 
\[ h_s^j(t) = \frac{1}{t} \psi(t) \hat{m}_{k_j}\left( \frac{s}{t} \right) \langle s \rangle^{\gamma + \delta} \]
and
\[ T_s^j = \frac{1}{2\pi} \langle s \rangle^{-(\gamma + \delta)} ( 1 + 2^{n_j} A)^{-\gamma} \exp(i2^{n_j}sA) . \]
Then the same calculation as in the proof of \eqref{equ-1-thm-main} yields
\[ \E \left\| \sum_{j} \epsilon_j \int_\R h_s^j(t) T_s^{n_j} ds f_j \right\|_{L^p(Y(\Lambda^\beta))} \leq \max_j \sup_{s \in \R} \|h_s^j\|_{\Lambda^\beta} \int_\R R \left( \left\{ T_s^n : \: n \in \Z \right\} \right) ds \E \left\| \sum_j \epsilon_j f_j \right\|_{L^p(Y)}. \]
But then we can copy from above that
\[ \max_j \sup_{s \in \R} \|h_s^j\|_{\Lambda^\beta} \lesssim \max_k \|m_k\|_{W^c_2(\R)} .\]
We have concluded the proof of \eqref{equ-2-thm-main}.
\end{proof}

\begin{remark}
\label{rem-main-measurability}
In Theorem \ref{thm-main} above, one might wonder whether $t \mapsto m(tA)f$ belongs \emph{a priori} to $L^p(Y(\Lambda^\beta))$.
That is, whether $t\mapsto m(tA)f(x,\omega)$ belongs to $\Lambda^\beta$ for a.e. $(x,\omega) \in \Omega \times \Omega'$, $(t,x,\omega) \mapsto m(tA)f(x,\omega)$ is strongly measurable $\R_+ \times \Omega \times \Omega' \to \C$ and $(x,\omega) \mapsto \norm{m(tA)f(x,\omega)}_{\Lambda^\beta}$ belongs to $L^p(Y)$.

That this is indeed the case can be seen with the following reasoning.
First note that for any $f \in L^p(Y)$, the function $\R_+ \to L^p(Y)$, $t \mapsto m(tA)f$ is continuous.
Since for every $B \subseteq \Omega \times \Omega'$ of finite measure, $L^p(\Omega,Y(\Omega')) \to L^1(B), \: g \mapsto g 1_B$ is continuous,
we obtain $\R_+ \to L^1(B)$, $t \mapsto m(tA)f 1_B$ continuous.
Exhausting $\Omega \times \Omega'$ by a sequence of such $B$, according to \cite[Proposition 1.2.25]{HvNVW}, we deduce that $\R_+ \times \Omega \times \Omega' \to \C, \: (t,x,\omega) \mapsto m(tA)f(x,\omega)$ is a strongly measurable function (by choosing for each $t \in \R_+$ the right representative of $m(tA)f(x,\omega)$).

In the following, we show that this strongly measurable function coincides a.e. with an element of $L^p(Y(\Lambda^\beta))$, thus finishing the proof.
Take first, as in the proof of Theorem \ref{thm-main}, $f \in D_A \subseteq L^p(Y)$.
Thus, there exists a $\phi \in C^\infty_c(\R_+)$ such that $f = \phi(A)f$, and consequently, ruling out the supports as already done in the proof of Theorem \ref{thm-main}, there exists $\psi \in C^\infty_c(\R_+)$ such that $m(tA)f=m(tA)\phi(A)f = \psi(t) m(tA)\phi(A)f = \psi(t)m(tA)f$.
Then the representation formula from \ref{lem-representation-formula-wave-operators} can be written as
\begin{equation}
\label{equ-1-rem-main-measurability}
\psi(t)m(tA)f = \frac{1}{2\pi} \int_\R \frac{1}{t} \psi(t) \hat{m}\left(\frac{s}{t}\right) \exp(isA)f ds = \int_\R h_s(t) f_s ds ,
\end{equation}
where $h_s(t) = \frac{1}{2 \pi} \frac{1}{t} \psi(t) \hat{m}\left(\frac{s}{t} \right)$ satisfies, as in the proof of Theorem \ref{thm-main}, $\|h_s\|_{\Lambda^\beta} \langle s \rangle^{\gamma + \delta} \lesssim 1$ and $f_s = \exp(isA)f$ satisfies $\int_\R \langle s \rangle^{-\gamma - \delta} \|f_s\|_{L^p(Y)} ds < \infty$.
Therefore,
\[ \int_\R \|h_s\|_{\Lambda^\beta} \|f_s\|_{L^p(Y)} ds < \infty .\]
Moreover, one checks elementarily that $s \mapsto h_s$ is continuous $\R \to \Lambda^\beta$ and moreover, $s \mapsto f_s$ is differentiable, hence continuous $\R \to L^p(Y)$.
Then the map $s \mapsto h_s \otimes f_s$ is continuous $\R \to L^p(Y(\Lambda^\beta))$ and therefore, the integral in \eqref{equ-1-rem-main-measurability} is a Bochner integral in the space $L^p(Y(\Lambda^\beta))$.
Denote its value in $L^p(Y(\Lambda^\beta))$ by $X$.
Consider the operator
\[ I_{a,b} : \: \begin{cases} \Lambda^\beta & \to \C \\ \xi & \mapsto \int_a^b \xi(t) dt \end{cases} \]
where $0 < a < b < \infty$.
Then $I_{a,b}$ is continuous, so $\Id_{L^p(Y)} \otimes I_{a,b} : L^p(Y(\Lambda^\beta)) \to L^p(Y)$ is also continuous.
We infer that
\begin{align*}
\MoveEqLeft
\Id_{L^p(Y)} \ot I_{a,b}(m(tA)f) = \int_a^b \int_\R h_s(t) f_s ds dt \overset{\text{Fubini}}= \int_\R \int_a^b h_s(t) dt f_s ds \\
& = \int_\R I_{a,b}(h_s) f_s ds = \int_\R \Id_{L^p(Y)} \ot I_{a,b}(h_s \ot f_s) ds \\
& = \Id_{L^p(Y)} \ot I_{a,b} \int_\R h_s \ot f_s ds = \Id_{L^p(Y)} \ot I_{a,b}(X)
\end{align*}
Here we used in the penultimate step that $\int_\R h_s \ot f_s ds$ is a Bochner integral in $L^p(Y(\Lambda^\beta))$ so that the bounded operator $\Id_{L^p(Y)} \ot I_{a,b}$ can be swapped with the integral.
We infer by \cite[Proposition 1.2.14]{HvNVW} that for a.e. $t > 0$, $m(tA)f = X(t,\cdot,\cdot)$ as equality in $L^p(Y)$.
Thus, for a.e. $t > 0$ and a.e. $(x,\omega), \: m(tA)f(x,\omega) = X(t,x,\omega)$ as equality in $\C$, and by a Fubini argument this yields equality for a.e. $(t,x,\omega) \in \R_+ \times \Omega \times \Omega'$ as we wanted.
So we can conclude that $t \mapsto m(tA)f$ indeed defines an element in $L^p(Y(\Lambda^\beta))$.

If $f \in L^p(Y)$ is arbitrary, then take by density of $D_A$ a sequence $(f_n)_n \subseteq D_A$ converging to $f$ in $L^p(Y)$.
We have $m(tA)f = \lim_n m(tA) f_n$ in $L^p(Y)$ for any $t > 0$.
Moreover, Theorem \ref{thm-main} for $f \in D_A$ yields that $t \mapsto m(tA) f_n$ is a Cauchy sequence in $L^p(Y(\Lambda^\beta))$, with limit, say, $X \in L^p(Y(\Lambda^\beta))$.
Consider again $I_{a,b}$ and argue similarly as above to deduce that $m(tA)f(x,\omega)$ equals a.e. on $\R_+ \times \Omega \times \Omega'$ with an element in $L^p(Y(\Lambda^\beta))$.

If $\beta > \frac12$, then $\Lambda^\beta$ consists of continuous functions.
Then replacing $I_{a,b}$ by $\delta_t : \Lambda^\beta \to \C, \: \xi \mapsto \xi(t)$, and arguing as above, we obtain that $m(tA)f(x,\omega) = X(t,x,\omega)$ for a.e. $(x,\omega)$.
Choosing then the representative $X(t,\cdot,\cdot)$ of $m(tA)f$ for each fixed $t > 0$, the function $t \mapsto m(tA)f(x,\omega)$ becomes continuous for a.e. $(x,\omega)$.
In the sequel, we shall tacitly always choose this representative.
\end{remark}

Using the dilation invariance structure of the $\Lambda^\beta_{2,2}(\R_+)$ norm, we can easily generalise Theorem \ref{thm-main} in the following corollary, where the compact support condition on the spectral multiplier $m$ is replaced by a summability condition of norms of dilates of $m$.

\begin{cor}
\label{cor-main-full-support}
Let $Y$ be a UMD lattice, $1 < p < \infty$ and $(\Omega,\mu)$ a $\sigma$-finite measure space.
Let $\beta \geq 0$.
Let $A$ be a $0$-sectorial operator on $L^p(Y)$.
Assume that $A$ has a $\Hor^\alpha_2$ calculus on $L^p(Y)$.
Pick as in Theorem \ref{thm-main}
\[ c > \alpha + \max \left( \frac12, \frac{1}{\type L^p(Y)} - \frac{1}{\cotype L^p(Y)} \right) + \frac12 + \beta . \]
Let $m$ be a spectral multiplier with $m(0) = 0$ such that for some dyadic partition $(\dyad_n)_{n \in \Z}$ of $\R_+$, we have $\sum_{n \in \Z} \|m(2^n \cdot) \dyad_0\|_{W^c_2(\R)} < \infty$.
Then
\begin{equation}
\label{equ-1-cor-main-full-support}
\| t \mapsto m(tA)f\|_{L^p(Y(\Lambda^\beta_{2,2}(\R_+)))} \leq C \sum_{n \in \Z} \|m(2^n \cdot) \dyad_0\|_{W^c_2(\R)} \|f\|_{L^p(Y)} .
\end{equation}
Moreover, let $(m_k)_{k \in \N}$ be a sequence of spectral multipliers with $m_k(0) = 0$ such that for any $k \in \N$, $\sum_{n \in \Z} \|m_k(2^n \cdot) \dyad_0\|_{W^c_2(\R)} < \infty$.
Let $(\omega^k(l))_{l \in \N}$ be the non-increasing rearrangement of the sequence $(\|m_k(2^n \cdot) \dyad_0\|_{W^c_2(\R)})_{n \in \Z}$.
If we have $\sum_{l \in \N} \sup_k \omega^k(l) < \infty$, then 
\begin{equation}
\label{equ-2-cor-main-full-support}
\left\| \left( \sum_k \| t \mapsto m_k(tA) f_k \|_{\Lambda^\beta_{2,2}(\R_+)}^2 \right)^{\frac12} \right\|_{L^p(Y)} \leq C \sum_{l \in \N} \sup_k \omega^k(l) \left\| \left( \sum_k | f_k |^2 \right)^{\frac12} \right\|_{L^p(Y)} .
\end{equation}
\end{cor}

\begin{proof}
Let us start with the proof of \eqref{equ-1-cor-main-full-support}.
We decompose, using the dilation invariance of the $\Lambda^\beta$ norm and then Theorem \ref{thm-main} for the function $(m \dyad_n)(2^{n}\cdot)$ with support in $[\frac12,2]$,
\begin{align*}
\| t \mapsto m(tA)f \|_{L^p(Y(\Lambda^\beta))} & = \left\| \sum_{n \in \Z} t \mapsto m(tA) \dyad_n(tA) f \right\|_{L^p(Y(\Lambda^\beta))} \\
& \leq \sum_{n \in \Z} \|t \mapsto m(tA) \dyad_n(tA)f\|_{L^p(Y(\Lambda^\beta))} \\
& = \sum_{n \in \Z} \| t \mapsto (m\dyad_n)(2^{n}tA) f\|_{L^p(Y(\Lambda^\beta))} \\
& \lesssim \sum_{n \in \Z} \|(m\dyad_n)(2^{n} \cdot)\|_{W^c_2(\R)} \|f\|_{L^p(Y)} \\
& = \sum_{n \in \Z} \|m(2^n \cdot) \dyad_0\|_{W^c_2(\R)} \|f\|_{L^p(Y)} .
\end{align*}
Here in the first equality we used $\sum_{n \in \Z} \dyad_n(t) = 1 \: (t > 0)$ together with Lemma \ref{lem-Hormander-convergence-lemma}.
For the proof of \eqref{equ-2-cor-main-full-support}, let for $k \in \N$, $n^k_{(\cdot)} : \N \to \Z$ be the bijection corresponding to the non-increasing rearrangement of $(\|m_k(2^n \cdot) \dyad_0\|_{W^c_2(\R)})_{n \in \Z}$.
Then according to \eqref{equ-2-thm-main},
\begin{align*}
\left\| \left( \sum_k \| t \mapsto m_k(tA) f_k \|_{\Lambda^\beta}^2 \right)^{\frac12} \right\|_{L^p(Y)} & = \left\| \left( \sum_k \left\| \sum_{l \in \N} t \mapsto m_k(tA) \dyad_{n^k_l}(tA) f_k \right\|_{\Lambda^\beta}^2 \right)^{\frac12} \right\|_{L^p(Y)} \\
& \leq \sum_{l \in \N} \left\| \left( \sum_k \left\| t \mapsto m_k(tA) \dyad_{n^k_l}(tA) f_k \right\|_{\Lambda^\beta}^2 \right)^{\frac12} \right\|_{L^p(Y)} \\
& = \sum_{l \in \N} \left\| \left( \sum_k \left\| t \mapsto m_k(2^{n^k_l}tA) \dyad_0(tA) f_k \right\|_{\Lambda^\beta}^2 \right)^{\frac12} \right\|_{L^p(Y)} \\
& \lesssim \sum_{l \in \N} \sup_k \|m_k(2^{n^k_l} \cdot) \dyad_0 \|_{W^c_2(\R)} \left\| \left( \sum_k |f_k|^2 \right)^{\frac12} \right\|_{L^p(Y)} .
\end{align*}
\end{proof}

\begin{remark}
In the setting of Corollary \ref{cor-main-full-support}, similarly to Remark \ref{rem-main-measurability}, we note that $t \mapsto m(tA)f$ belongs \emph{a priori} to $L^p(Y(\Lambda^\beta))$ thanks to absolute convergence of $\sum_{n \in \Z} m(tA) \dyad_n(tA)f$ in $L^p(Y(\Lambda^\beta))$ by assumption of Corollary \ref{cor-main-full-support}, and convergence of that series for fixed $t > 0$, in $L^p(Y)$, to $m(tA)f$, according to Lemma \ref{lem-Hormander-convergence-lemma}.
Indeed, use again $\Id_{L^p(Y)} \otimes I_{a,b}$ as in the proof of Remark \ref{rem-main-measurability}.
\end{remark}

In the next corollary, we break down the $\Lambda^\beta$ norm from Corollary \ref{cor-main-full-support} to the more classical supremum norm.
Note that $\sup_{t > 0} |m(tA)f(x,\omega)| = \sup_{t > 0, \: t \in \mathbb Q} |m(tA)f(x,\omega)|$ is measurable as a supremum of countably many measurable functions, where equality follows from the fact that $t \mapsto m(tA)f(x,\omega)$ belongs to $\Lambda^\beta$ and hence is continuous, for a.e. $(x,\omega)$.

\begin{cor}
\label{cor-main-C0}
Let $Y$ be a UMD lattice, $1 < p < \infty$ and $(\Omega,\mu)$ a $\sigma$-finite measure space.
Let $A$ be a $0$-sectorial operator on $L^p(Y)$.
Assume that $A$ has a $\Hor^\alpha_2$ calculus on $L^p(Y)$.
Choose some
\[ c > \alpha + \max \left( \frac12, \frac{1}{\type L^p(Y)} - \frac{1}{\cotype L^p(Y)} \right) + 1 . \]
Let $m$ be a spectral multiplier such that for some dyadic partition $(\dyad_n)_{n \in \Z}$ of $\R_+$, we have $\sum_{n \in \Z} \|m(2^n \cdot) \dyad_0\|_{W^c_2(\R)} < \infty$.
Then $t \mapsto m|_{(0,\infty)}(tA)f$ belongs to $L^p(Y(C_0(\R_+)))$ and
\[ \| \sup_{t > 0} |m(tA)f| \|_{L^p(Y)} \leq C \left( |m(0)| + \sum_{n \in \Z} \|m(2^n \cdot) \dyad_0\|_{W^c_2(\R)} \right) \|f\|_{L^p(Y)} . \]
Moreover, let $(m_k)_{k \in \N}$ be a sequence of spectral multipliers such that $\sup_k |m_k(0)| < \infty$ and $\sum_{l \in \N} \sup_k \omega^k(l) < \infty$, where the $\omega^k$ are the non-increasing rearrangements of $(\|m_k(2^n \cdot)\dyad_0\|_{W^c_2(\R)})_{n \in \Z}$ as in Corollary \ref{cor-main-full-support}.
Then
\[ \left\| \left( \sum_k \sup_{t > 0} |m_k(tA)f_k|^2 \right)^{\frac12} \right\|_{L^p(Y)} \leq C \left( \sup_k |m_k(0)| + \sum_{l \in \N} \sup_k \omega^k(l) \right) \left\| \left( \sum_k |f_k|^2 \right)^{\frac12} \right\|_{L^p(Y)} . \]
\end{cor}

\begin{proof}
This corollary follows from Corollary \ref{cor-main-full-support} noting that for $\beta > \frac12$, the classical Sobolev embedding yields $W^\beta_2(\R) \hookrightarrow C_0(\R)$, so that passing to $\R \leadsto \R_+$ via the exponential function, we have $\Lambda^\beta_{2,2}(\R_+) \hookrightarrow C_0(\R_+)$.
Then $L^p(Y(\Lambda^\beta_{2,2}(\R_+))) \hookrightarrow L^p(Y(C_0(\R_+)))$, since $L^p(Y)$ is a lattice.
Therefore, if $P : L^p(Y) \to L^p(Y)$ denotes the projection onto the null-space of $A$ and $m_1(\lambda) = 1_{(0,\infty)}(\lambda) m(\lambda)$, we have
$m(tA)f = m(tA)Pf + m(tA)(\Id - P)f = m(0)Pf + m_1(tA)f$.
Thus, $\| \sup_{t > 0}|m(tA)f|  \|_{L^p(Y)} \leq |m(0)| \|Pf\|_{L^p(Y)} + \|\sup_{t > 0} |m_1(tA)f|\|_{L^p(Y)} \lesssim |m(0)| \|f\|_{L^p(Y)} + \| t \mapsto m_1(tA) f \|_{L^p(Y(\Lambda^\beta))}.$
Now apply Corollary \ref{cor-main-full-support} to the last summand.
The proof of the second part of the corollary is similar.
\end{proof}

In the following corollary, we obtain pointwise convergence for dilates of spectral multipliers $m(tA)f(x,\omega)$.
Note hereby that we need that the orbits $t \mapsto m|_{(0,\infty)}(tA)f(x,\omega)$ belong to $C_0(\R_+)$, not only to $L^\infty(\R_+)$ (see Corollary \ref{cor-main-C0} above).

\begin{cor}
\label{cor-pointwise-convergence}
Under the hypotheses of Corollary \ref{cor-main-C0}, for any $f \in L^p(Y) = L^p(\Omega,Y(\Omega'))$ and any spectral multiplier $m$ such that $\sum_{n \in \Z} \|m(2^n \cdot) \dyad_0\|_{W^c_2(\R)} < \infty$, we have for pointwise a.e. $(x,\omega) \in \Omega \times \Omega'$ that
\begin{align*}
m(tA)f(x,\omega) & \to m(0) Pf(x,\omega) & \quad (t \to 0+), \\
m(tA)f(x,\omega) & \to m(0) Pf(x,\omega) & \quad (t \to \infty),
\end{align*}
where $P : L^p(Y) \to L^p(Y)$ denotes the projection onto the null-space of $A$.
The convergence holds also pointwise for a.e. $x \in \Omega$ in $Y$ and in $L^p(Y)$.
\end{cor}

\begin{proof}
We decompose $m(\lambda) = m(\lambda)1_{(0,\infty)}(\lambda) + m(0)1_{\{0\}}(\lambda) =: m_1(\lambda) + m(0) 1_{\{0\}}(\lambda)$.
Inserting the functional calculus yields
\[ m(tA)f(x,\omega) = m_1(tA)f(x,\omega) + m(0) Pf(x,\omega) \quad (t > 0,\: x \in \Omega,\: \omega \in \Omega') . \]
According to Corollary \ref{cor-main-C0}, $m_1(tA)f$ belongs to $L^p(Y(C_0(\R_+)))$, so by the definition of vector valued lattices as in Subsection \ref{subsec-UMD-lattices}, we deduce that for almost every $(x,\omega) \in \Omega \times \Omega'$, $m_1(tA)f(x,\omega)$ belongs to $C_0(\R_+)$, and thus we have
\[\lim_{t \to 0+}m_1(tA)f(x,\omega) = \lim_{t \to \infty} m_1(tA)f(x,\omega) = 0.\]
Then the last statement follows since $Y(c_0) \hookrightarrow c_0(Y)$ and $L^p(Y(c_0)) \hookrightarrow c_0(L^p(Y))$.
\end{proof}

\begin{remark}
\label{rem-main-interpolation}
In Theorem \ref{thm-main} above, in case that $A$ acts on $L^q(\Omega)$ for the scale $1 < q < \infty$ in a consistent way such that $A \otimes \Id_Z$ has a H\"ormander calculus $\Hor^\alpha_2$ on $L^q(Z)$ for any $1 < q < \infty$ and any UMD lattice $Z$ with a uniform value of $\alpha$, and moreover, $A$ is self-adjoint on $L^2(\Omega)$, then the parameter $c$ in Theorem \ref{thm-main} can be improved by complex interpolation.
\begin{enumerate}
\item
Namely, let $\theta \in (0,1)$ be a parameter and let $p \in (1,\infty)$ satisfy $|\frac1p - \frac12| < \frac{\theta}{2}$.
Moreover, let $Y$ be a UMD lattice realized over the measure space $\Omega'$, satisfying $Y = [H,Y_0]_\theta$ for some Hilbert space $H = L^2(\Omega',\mu')$ and a further UMD lattice $Y_0 = Y_0(\Omega')$.
This property and the parameter $\theta$ can be expressed in terms of $p$-convexity and $q$-concavity of $Y$ \cite[Lemma 2.11]{DKK}.
Then if $\supp(m) \subseteq [\frac12,2]$, we have
\[ \| t \mapsto m(tA) f \|_{L^p(Y(\Lambda^\beta))} \leq C \|m\|_{W^c_2(\R)} \|f\|_{L^p(Y)} \]
for
\[ c > \beta + \theta \left( \alpha + \frac32 \right) .\]
Moreover, if $(m_k)_{k \in \N}$ is a sequence of spectral multipliers with $\supp(m_k) \subseteq [\frac12,2]$, we have
\[ \left\| \left( \sum_k \|t \mapsto m_k(tA)f_k\|_{\Lambda^\beta}^2 \right)^{\frac12} \right\|_{L^p(Y)} \leq C \sup_k \|m_k\|_{W^c_2(\R)} \left\| \left( \sum_k |f_k|^2 \right)^{\frac12} \right\|_{L^p(Y)} . \]
Finally, under the above hypotheses, the same exponent $c$ can be chosen in Corollary \ref{cor-main-full-support} and, with $\beta$ replaced by $\frac12$, in Corollary \ref{cor-main-C0}.
We refer to Section \ref{sec-examples} for examples of such operators $A$.
\item
In the scalar case $Y = \C$, if $A$ has $\Hor^\alpha_2$ calculus with order $\alpha > \frac{d}{2}$, where typically, $d$ is a dimension of $\Omega$, our parameter becomes
\begin{equation}
\label{equ-rem-main-interpolation}
c  > \beta + (d+3) \left| \frac1p - \frac12 \right| .
\end{equation}
\item As a simple example, if $A$ is as in 2., then $\| t \mapsto m(tA)f \|_{L^p(Y(\Lambda^\beta))} \lesssim_\epsilon C_m \|f\|_{L^p(Y)}$ provided that $m$ is of class $C^c(0,\infty)$ with $c \in \N$ satisfying \eqref{equ-rem-main-interpolation}, $m$ vanishes on $[0,1]$ and for some $\epsilon > 0$,
\[ \left| \lambda^k \frac{d^k}{d\lambda^k}m(\lambda) \right| \leq C_m \lambda^{-\epsilon} \quad ( k = 0,1,\ldots,c,\: \lambda > 1).\]
\end{enumerate}
\end{remark}

\begin{proof}
1. We consider the mapping
\begin{equation}
\label{equ-1-rem-main-interpolation}
\begin{cases}
W^c_2(\frac12,2) \times L^q(Z) & \to L^q(Z(\Lambda^\beta)) \\
(m,f) & \mapsto (t\mapsto m(tA)f).
\end{cases}
\end{equation}
It is bounded for $q = 2$, $Z = H$ any Hilbert space and
\begin{equation}
\label{equ-2-rem-main-interpolation}
c = c_0 = \beta .
\end{equation}
Indeed, for $m \in W^\beta_2(\frac12,2) \subseteq \Lambda^\beta$ the function $\lambda \mapsto m(\lambda \cdot)$ is continuous and bounded $\R_+ \to \Lambda^\beta$ \cite[p.~148]{MauMeda}.
Thus, by self-adjointness of $A$, we have that $f \mapsto (t \mapsto m(tA)f)$ is bounded $L^2(H) \to L^2(H(\Lambda^\beta))$ with norm bounded by $\sup_{\lambda > 0} \|m(\lambda \cdot)\|_{\Lambda^\beta} = \|m\|_{\Lambda^\beta} \cong \|m\|_{W^c_2(\R)}$, the last equivalence according to Remark \ref{rem-Lambda-dilation-invariant}.
On the other hand, according to the assumptions of the Remark, \eqref{equ-1-rem-main-interpolation} is also bounded for any $q \in (1,\infty)$, $Z = Y_0$ and
\begin{equation}
\label{equ-3-rem-main-interpolation}
c = c_1 > \alpha + 1 + \frac12 + \beta .
\end{equation}
Here, the second summand is $1 > \max\left( \frac12, \frac{1}{\type L^q(Z)} - \frac{1}{\cotype L^q(Z)} \right)$.
Bilinear complex interpolation \cite[4.4.1 Theorem]{BeL} of \eqref{equ-1-rem-main-interpolation} yields boundedness for the choice of spaces $W^c_2(\frac12,2) = [W^{c_0}(\frac12,2),W^{c_1}(\frac12,2)]_\theta$, $L^p(Y) = [L^2,L^q]_\theta([H,Y_0]_\theta) = [L^2(H),L^q(Y_0)]_\theta$ \cite[Theorem 5.1.2]{BeL} with appropriate $q$ close to $1$ or $\infty$ and similarly, $L^p(Y(\Lambda^\beta)) = [L^2(H(\Lambda^\beta)),L^q(Y_0(\Lambda^\beta))]_\theta$.
Here, \eqref{equ-2-rem-main-interpolation} and \eqref{equ-3-rem-main-interpolation} give
\[ c = (1-\theta)c_0 + \theta c_1 > \beta + \theta \left( \alpha + \frac32 \right) . \]
For the square function estimate, we argue similarly as before and interpolate the bilinear mapping
\[ \begin{cases} \ell^\infty_N(W^c_2(\frac12,2)) \times L^q(Z(\ell^2_N)) & \to L^q(Z(\ell^2_N(\Lambda^\beta))) \\ (m_1, \ldots, m_N, (f_k)_{k = 1}^N) & \mapsto (t \mapsto m_k(tA)f_k)_{k = 1}^N . \end{cases} \]
Then we use Remark \ref{rem-R-bdd-Hilbert} to deduce that for $q = 2$, $Z = H$, $c_0 = \beta$ is allowed as in the first case of the proof.

2. We can take $\theta = 2\left|\frac1p - \frac12 \right|$ and apply 1.

3. We want to apply part 2. together with \eqref{equ-1-cor-main-full-support}.
Note that for $n \leq -1$, by the support condition on $m$, we have $m(2^n\cdot)\dyad_0 = 0$.
On the other hand, for $n \geq 0$, we have
\begin{align*}
\|m(2^n\cdot)\dyad_0\|_{W^c_2(\R)} & \lesssim \max_{k = 0,1, \ldots, c}\left\|\frac{d^k}{d\lambda^k} (m(2^n\lambda)) \right\|_{L^2(\frac12,2)} \\
& \lesssim \max_{k = 0,1,\ldots,c} 2^{kn} \left\| \frac{d^k}{d\lambda^k}(m)(2^n\lambda) \right\|_{L^2(\frac12,2)} \\
& \lesssim C_m 2^{-n\epsilon},
\end{align*}
which is summable over $n \geq 0$ for $\epsilon > 0$.
\end{proof}

\begin{remark}
\label{rem-Mau-Meda}
If $A$ is a left invariant sublaplacian on a stratified Lie group $G$, then $A$ does have a $\Hor^\alpha_2$ calculus on $L^p(G,Y)$ for any UMD lattice $Y$, with $\alpha$ depending on the doubling dimension of $G$ and concavity and convexity exponents of $Y$  (see Subsection \ref{subsec-examples-GE}).
It also has a $\Hor^\alpha_2$ calculus on $L^p(G)$ for $\alpha > Q/2$, where $Q$ is the homogeneous dimension of $G$ \cite{Christ}.
It is essentially shown in \cite[Theorem 2.1 (ii)]{MauMeda}, that Corollary \ref{cor-main-full-support} holds for this choice of $A$, $Y = \C$ and $c  > Q \left| \frac1p - \frac12 \right| + \beta$, compared to our $c > (Q+3) \left| \frac1p - \frac12 \right| + \beta$ using Remark \ref{rem-main-interpolation}.
Thus, Corollary \ref{cor-main-full-support} extends \cite[Theorem 2.1 (ii)]{MauMeda} to the vector valued case of $L^p(Y)$ and to square function estimates, under the price of a higher differentiation order $c$.
Another related result in the case $Y = \C$ with a different summability condition of $\|m(2^n \cdot)\dyad_0\|_{W^c_2(\R)}$ is given in \cite[Theorem 1.2]{Choi}.
\end{remark}

In Corollary \ref{cor-main-C0}, if the spectral multiplier $m(t)$ converges to a non-zero value at $t \to 0+$, then
\[ \liminf_{n \to - \infty} \|m(2^n \cdot) \dyad_0 \|_{W^c_2(\R)} \geq \liminf_{n \to - \infty} \|m(2^n \cdot) \dyad_0 \|_{L^2(\R)} \gtrsim \lim_{t \to 0+}|m(t)| > 0 .\]
Thus, the series $\sum_{n \leq 0} \|m(2^n \cdot) \dyad_0 \|_{W^c_2(\R)}$ has no chance to converge.
However, a possibility to include such spectral multipliers is given in the following Proposition \ref{prop-main-exp}, using a semigroup maximal estimate when available.

\begin{prop}
\label{prop-main-exp}
Let $Y = Y(\Omega')$ be a UMD lattice, $1 < p < \infty$ and $(\Omega,\mu)$ a $\sigma$-finite measure space.
Let $A$ be a $0$-sectorial operator on $L^p(Y)$.
Assume that $A$ has a $\Hor^\alpha_2$ calculus on $L^p(Y)$.
In addition, assume that $A$ is of the form $A = A_0 \otimes \Id_Y$, where $A_0$ is $0$-sectorial on $L^p(\Omega)$ and that the semigroup $\exp(-tA_0)$ generated by $-A_0$ is lattice positive and contractive on $L^p(\Omega)$ (more generally regular contractive on $L^p(\Omega)$).
Choose some
\[ c > \alpha + \max \left( \frac12, \frac{1}{\type L^p(Y)} - \frac{1}{\cotype L^p(Y)} \right) + 1 . \]
Let $m$ be a spectral multiplier satisfying
\[ m|_{[0,1]} \in C^1[0,1] , \: \| m' \cutoff_0 \|_{\Hor^{c-1}_2} < \infty, \: \sum_{n \geq 0} \| m(2^n \cdot) \dyad_0 \|_{W^c_2(\R)} < \infty , \]
where $\cutoff_0 \in C^\infty(\R_+)$ equals $0$ on $[4,\infty)$ and equals $1$ on $(0,2]$, and $(\dyad_n)_{n \in \Z}$ is, as before, a dyadic partition of $\R_+$.
Then for almost every $(x,\omega) \in \Omega \times \Omega'$, $t \mapsto m(tA)f(x,\omega)$ belongs to $L^\infty(\R_+)$ and
\begin{equation}
\label{equ-1-prop-main-exp}
 \left\| \sup_{t > 0} |m(tA) f| \right\|_{L^p(Y)} \leq C \left( |m(0)| + \|m' \cutoff_0 \|_{\Hor^{c-1}_2} + \sum_{ n \geq 0} \| m(2^n \cdot) \dyad_0 \|_{W^c_2(\R)} \right) \|f\|_{L^p(Y)} .
\end{equation}
Moreoever, let $(m_k)_{k \in \N}$ be a sequence of spectral multipliers satisfying
\[ m_k|_{[0,1]} \in C^1[0,1], \: \sup_k |m_k(0)| < \infty, \: \sup_k \|m_k' \cutoff_0 \|_{\Hor^{c-1}_2} < \infty, \: \sum_{n \geq 0} \sup_k \|m_k(2^n \cdot) \dyad_n \|_{W^c_2(\R)} < \infty . \]
Then
\begin{align}
\label{equ-2-prop-main-exp}
& \left\| \left( \sum_k \sup_{t > 0} |m_k(tA) f_k|^2 \right)^{\frac12} \right\|_{L^p(Y)} \\
& \leq C \left( \sup_k |m_k(0)| + \sup_k \|m_k' \cutoff_0 \|_{\Hor^{c-1}_2} + \sum_{ n \geq 0} \sup_k \|m_k(2^n \cdot) \dyad_0 \|_{W^c_2(\R)} \right) \left\| \left( \sum_k |f_k|^2 \right)^{\frac12} \right\|_{L^p(Y)} . \nonumber
\end{align}
\end{prop}

\begin{proof}
We start by proving \eqref{equ-1-prop-main-exp}.
Let us put $n(t) = m(t) - m(0) \exp(-t)$.
We then have $\| \sup_{t > 0} |m(tA)f| \|_{L^p(Y)} \leq \| \sup_{t > 0} |n(tA)f| \, \|_{L^p(Y)} + |m(0)| \, \| \sup_{t > 0} |e^{-tA} f| \, \|_{L^p(Y)}$.
The second summand is bounded by $C |m(0)| \, \|f\|_{L^p(Y)}$ according to \cite[Theorem 2]{Xu2015} (note that since $A$ has a $\Hor^\alpha_2$ calculus, $\exp(-tA_0)$ extends to a bounded analytic semigroup on any strict subsector of $\C_+$).
Invoking Corollary \ref{cor-main-C0}, it suffices to estimate
\[ \sum_{l \in \Z} \|n(2^l \cdot)\dyad_0\|_{W^c_2(\R)}  < \infty .\]
We split according to $l \geq 0$ or $l \leq 0$.
For $l \geq 0$, we have $\|n(2^l \cdot) \dyad_0\|_{W^c_2(\R)} \leq \|m(2^l \cdot) \dyad_0 \|_{W^c_2(\R)} + |m(0)| \, \|\exp(-2^l (\cdot) ) \dyad_0 \|_{W^c_2(\R)} \lesssim \|m(2^l \cdot) \dyad_0 \|_{W^c_2(\R)} + |m(0)| 2^{l\lceil c\rceil}\exp(-\frac12 \cdot 2^l)$, which is summable according to the hypotheses.
For $l \leq 0$, we note by some elementary manipulation that $\|n(2^l \cdot) \dyad_0 \|_{W^c_2(\R)} \lesssim \|n(2^l \cdot) \dyad_0 \|_{L^2(\R)} + \|n(2^l \cdot)' \widetilde{\dyad_0}\|_{W^{c-1}_2(\R)}$, where $\widetilde{\dyad_0} \in C^\infty_c(\R_+)$ is some function equal to $1$ on the support of $\dyad_0$ and having support in $[\frac13,3]$.
For the first term in this estimate, we have
\begin{align*}
\|n(2^l\cdot)\dyad_0\|_{L^2(\R)}^2 & \lesssim \int_{\frac12}^2 |n(2^lt)|^2 dt = \int_{\frac12}^2 |m(2^lt) - m(0)\exp(-2^lt)|^2 dt \\
& = \int_{\frac12}^2 \left| m(0) + \int_0^{2^l t} m'(\xi) d\xi - m(0)\exp(-2^l t) \right|^2 dt \\
& \lesssim \int_{\frac12}^2 |m(0)(1-\exp(-2^l t))|^2 dt + \int_{\frac12}^2 \left| \int_0^{2^l t} m'(\xi) d\xi \right|^2 dt \\
& \lesssim |m(0)|^2 \int_{\frac12}^2 |2^l t|^2 dt + \int_{\frac12}^2 2^l t \int_0^{2^l t} |m'(\xi)|^2 d\xi dt \\
& \lesssim  2^{2l}|m(0)|^2 + 2^{2l} \|m'\|_{L^\infty(0,1)}^2 \\
& \lesssim  2^{2l}|m(0)|^2 + 2^{2l} \|m'\cutoff_0\|_{L^\infty(\R_+)}^2 \\
& \lesssim 2^{2l}(|m(0)|^2  + \|m'\cutoff_0\|_{\Hor^{c-1}_2}^2 ).
\end{align*}
We infer that $\sum_{l \leq 0} \|n(2^l \cdot)\dyad_0\|_{L^2(\R)} < \infty$.
For the second term, we have
\begin{align*}
\|n(2^l \cdot)' \widetilde{\dyad_0}\|_{W^{c-1}_2(\R)} & = 2^l \|n'(2^l \cdot) \widetilde{\dyad_0}\|_{W^{c-1}_2(\R)} = 2^l \|n'(2^l \cdot) \cutoff_0(2^l \cdot) \widetilde{\dyad_0}\|_{W^{c-1}_2(\R)} \\
& \lesssim 2^l \|n' \cutoff_0\|_{\Hor^{c-1}_2} \\
& \leq 2^l \left( \|m' \cutoff_0\|_{\Hor^{c-1}_2} + |m(0)| \, \underset{ < \infty}{\underbrace{\|e^{-t} \cutoff_0\|_{\Hor^{c-1}_2}}} \right),
\end{align*}
which is again summable over $l \leq 0$.

Now for the proof of \eqref{equ-2-prop-main-exp}, we proceed in a similar manner.
Here we use that $t\mapsto \exp(-tA)$ admits a maximal estimate on the UMD lattice $L^p(Y(\ell^2))$ according to \cite[Theorem 2]{Xu2015}.
Note that in contrast to Corollary \ref{cor-main-C0}, we cannot invoke the nonincreasing rearrangement of $(\|m_k(2^n \cdot)\dyad_0\|_{W^c_2(\R)})_{n \geq 0}$ due to the use of the semigroup.
\end{proof}

\begin{remark}
In \cite[Theorem 3.1]{Wro}. there is an easier proof of \eqref{equ-1-prop-main-exp} applying in most of the interesting cases of operators $A$ (the statement there concerns only scalar valued $L^p(\Omega)$ spaces, but seems verbatim to translate to the UMD lattice valued space case).
Wr\'obel's proof uses imaginary powers in place of the wave operators $\exp(isA)$ as they appear in our proof of Theorem \ref{thm-main}, and he imposes a different norm on $m$, which is slightly bigger in its derivation exponent, that is $c = \alpha + 2$, (note that in case of scalar valued $L^p(\Omega)$ spaces, we always have $\max\left( \frac12, \frac{1}{\type L^p(\Omega)} - \frac{1}{\cotype L^p(\Omega)} \right) = \frac12$, so that our $c > \alpha + \frac32$).
The exact norm in \cite{Wro} is uncomparable to our expression in \eqref{equ-1-prop-main-exp}, since the integration exponent in \cite[$C(n,\beta)$, p.~146]{Wro} is $q = 1$, whereas we have $q = 2$ in our space $W^c_q(\R) = W^c_2(\R)$.
Note that \cite{Wro} does not yield the pointwise convergence from Corollary \ref{cor-pointwise-convergence} nor square function estimates as in \eqref{equ-2-prop-main-exp}; but a sort of Paley-Littlewood equivalence in \cite[Section 4]{Wro}.
\end{remark}

Proposition \ref{prop-main-exp} applies in the following two important cases.
The second part concerns the Bochner-Riesz maximal operator.
Bounds for this operator in the case of the euclidean Laplacian have been obtained in dimension $2$, $1 < p < 2$ and certain exponents $\gamma$ by Tao \cite{Tao}, and in dimension $d \geq 3$ for sufficiently large $p$ and optimal exponent $\gamma > \max( d | \frac1p - \frac12| - \frac12, 0)$ by Lee \cite{Lee}.
Moreover, Seeger \cite{See} gave $L^p$ estimates of Paley-Littlewood $g$ functions associated with the Bochner-Riesz means.

\begin{cor}
\label{cor-wave-Bochner-Riesz-maximal}
Assume that the hypotheses of Proposition \ref{prop-main-exp} hold.
Let $c$ as in this proposition and
\[\delta > c  > \alpha + \max \left( \frac12, \frac{1}{\type L^p(Y)} - \frac{1}{\cotype L^p(Y)} \right) + 1. \]
Then the wave operators associated with $A$ satisfy the maximal estimate
\[ \left\| \sup_{t > 0} \left| (1  + t A)^{-\delta} \exp(itA) f \right| \, \right\|_{L^p(Y)} \leq C \| f \|_{L^p(Y)} . \]
Moreover, let 
\[ \gamma > c - \frac12 > \alpha + \max \left( \frac12, \frac{1}{\type L^p(Y)} - \frac{1}{\cotype L^p(Y)} \right) 
+ \frac12 . \]
Then the Bochner-Riesz means associated with $A$ satisfy the maximal estimate
\[ \left\| \sup_{t > 0} \left| \left( 1 - \frac{A}{t} \right)_+^\gamma f \right| \, \right\|_{L^p(Y)} \leq C \| f \|_{L^p(Y)} . \]
\end{cor}

\begin{proof}
We check the hypotheses of Proposition \ref{prop-main-exp} for the spectral multiplier functions $m(\lambda) = (1+\lambda)^{-\delta} \exp(i\lambda)$ and $m(\lambda) = (1 - 4 \lambda)_+^\gamma$, where the factor $4$ is for convenience.
For the wave spectral multiplier, we have $m(0) = 1$.
Moreover, if $c \in \N$, then $\frac{d^c}{d\lambda^c} m(\lambda) = \sum_{k = 0}^c \alpha_k (1 + \lambda)^{-\delta-k} e^{i\lambda}$ for certain coefficients $\alpha_k \in \C$.
Thus, $\frac{d^c}{d\lambda^c} m(2^n \cdot)|_\lambda = 2^{nc} \sum_{k =0}^c \alpha_k (1 + 2^n \lambda)^{-\delta-k} e^{i2^n \lambda}$ and $\|m(2^n \cdot)\dyad_0\|_{W^c_2(\R)} \lesssim 2^{nc} 2^{-n\delta}$.
By complex interpolation of the spaces $W^c_2$, we deduce that for $c > \frac12$, we have $\|m(2^n \cdot) \dyad_0 \|_{W^c_2(\R)} \lesssim 2^{nc} 2^{-n\delta}$ (see also \cite[p.~65-67]{KrPhD}).
We infer that for $\delta > c$, $\sum_{n \geq 0} \|m(2^n \cdot) \dyad_0 \|_{W^c_2(\R)} < \infty$.
Note that $m$ is a $C^\infty$ function in a neighborhood of $[0,4]$, so that $m' \cutoff_0 \in \Hor^{c-1}_2$.

For the Bochner-Riesz spectral multiplier, we have $m(0) = 1$, and since $\supp(m) \subseteq [0,\frac14]$, $m'\cutoff_0(\lambda) = m'(\lambda) = -4 \gamma (1 - 4 \lambda)^{\gamma - 1}_+$ in distributional sense.
According to \cite[p.~11]{COSY}, $m'$ belongs to $\Hor^{c-1}_2$ iff $\gamma - 1 > c - 1 - \frac12$.
For $n \geq 0$, note that $m(2^n \cdot)$ and $\dyad_0$ have disjoint supports, so that $\sum_{n \geq 0} \|m(2^n \cdot) \dyad_0\|_{W^c_2(\R)} = 0$.
\end{proof}

\begin{remark}
\label{rem-pointwise-convergence-exp}
Similarly to Corollary \ref{cor-pointwise-convergence}, under the hypotheses of Proposition \ref{prop-main-exp}, we obtain a pointwise convergence for $f \in L^p(Y)$ and for a.e. $(x,\omega) \in \Omega \times \Omega'$:
\begin{align*}
m(tA)f(x,\omega) & \to m(0) f(x,\omega) \quad (t \to 0+) , \\
m(tA)f(x,\omega) & \to m(0) Pf(x,\omega) \quad (t \to \infty),
\end{align*}
where $P : L^p(Y) \to L^p(Y)$ denotes the projection onto the null-space of $A$.
In particular, according to the proof of Corollary \ref{cor-wave-Bochner-Riesz-maximal}, we have with $\delta$ and $\gamma$ as in this corollary, for $f \in L^p(Y)$ and for a.e. $(x,\omega) \in \Omega \times \Omega'$:
\begin{align*}
(1+tA)^{-\delta} \exp(itA)f(x,\omega) & \to f(x,\omega) \quad (t \to 0+), \\
(1+tA)^{-\delta} \exp(itA)f(x,\omega) & \to Pf(x,\omega) \quad (t \to \infty),\\
(1 - tA)^\gamma_+f(x,\omega) & \to f(x,\omega) \quad (t \to 0+), \\
(1 - tA)^\gamma_+f(x,\omega) & \to Pf(x,\omega) \quad (t \to \infty).
\end{align*}
Again the convergence also holds pointwise for a.e. $x \in \Omega$ in $Y$ and in $L^p(Y)$.

Indeed, we decompose as in the proof of Proposition \ref{prop-main-exp}, $m(tA) = m(0) e^{-tA} + n(tA)$, where $n$ satisfies the hypotheses for $m$ in Corollary \ref{cor-pointwise-convergence} with $n(0) = 0$.
Now we have $\lim_{t \to 0+} m(tA)f(x,\omega) = \lim_{t \to 0+} m(0)e^{-tA}f(x,\omega) + \lim_{t \to 0+}n(tA)f(x,\omega) = m(0)f(x,\omega) + 0$, according to the pointwise convergence of the semigroup to the identity from \cite[Corollary 6.2]{HoMa}.
In the same way, we have $\lim_{t \to \infty} m(tA)f(x,\omega) = \lim_{t \to \infty} m(0)e^{-tA}f(x,\omega) + \lim_{t \to \infty}n(tA)f(x,\omega) = m(0)Pf(x,\omega) + 0$ \cite[Corollary 6.2]{HoMa}.
\end{remark}

\section{Examples and Applications}
\label{sec-examples}

In this section, we give several examples of $0$-sectorial operators $A$ and of UMD lattices $Y$, to illustrate our main results from Section \ref{sec-main}.

\subsection{H\"ormander calculus for (generalised) Gaussian estimates}
\label{subsec-examples-GE}

In all our results in Section \ref{sec-main}, we assume our $0$-sectorial operator $A$ to have a H\"ormander functional calculus on $L^p(Y)$, where $Y$ is a UMD lattice.
Such a functional calculus has been recently established for a broad class of differential operators in \cite{DKK}, for which such a calculus has been known before on the scalar valued spaces $L^p(\Omega,\C)$.
In this subsection, we present the wide range of such examples from \cite[Section 5]{DKK}.

Let us first recall the definition of space of homogeneous type.

\begin{defi}
Let $(\Omega,\dist,\mu)$ be a metric measure space, that is, $\dist$ is a metric on $\Omega$ and $\mu$ is a Borel measure on $\Omega$.
We denote $B(x,r) = \{ y \in \Omega :\: \dist(x,y) \leq r \}$ the closed balls of $\Omega$.
We assume that $\mu(B(x,r)) \in (0,\infty)$ for any $x \in \Omega$ and $r > 0$.
Then $\Omega$ is said to be a space of homogeneous type if there exists a constant $C < \infty$ such that the doubling condition holds:
\[ \mu(B(x,2r)) \leq C \mu(B(x,r)) \quad (x \in \Omega,\: r > 0) . \]
\end{defi}

We write in short $V(x,r) = \mu(B(x,r))$.
In what follows in this subsection, $(\Omega,\dist,\mu)$ is always a space of homogeneous type.
It is well-known that there exists some finite $d \in (0,\infty)$ such that $V(x,\lambda r) \leq C \lambda^d V(x,r)$ for any $x \in \Omega$, $r > 0$ and $\lambda \geq 1$.
Such a $d$ is called (homogeneous) dimension of $\Omega$.

We now introduce both the notions of Gaussian estimates and generalised Gaussian estimates.

\begin{defi}
Let $(T_t)_{t \geq 0}$ be a semigroup acting on $L^2(\Omega)$.
Assume that \[T_tf(x) = \int_\Omega p_t(x,y) f(y) \,dy\] for any $f \in L^2(\Omega),\:x \in \Omega,\:t > 0$ and some measurable functions $p_t : \Omega \times \Omega \to \C$.
Let $m \geq 2$.
Then $(T_t)_t$ is said to satisfy Gaussian estimates (of order $m$) if there exist constants $C,c > 0$ such that
\begin{equation}
\label{equ-GE-prelims}
|p_t(x,y)| \leq C \frac{1}{V(x,r_t)} \exp\Biggl(-c \biggl(\frac{\dist(x,y)}{r_t}\biggr)^{\frac{m}{m-1}} \Biggr) \quad (x,y \in \Omega,\: t > 0),
\end{equation}
where $r_t = t^{\frac1m}$.
%We say that $(T_t)_t$ satisfies two-sided Gaussian estimates if $p_t(x,y) \geq 0$ for all $t > 0,\:x ,y \in \Omega$, \eqref{equ-GE-%prelims} holds and in addition
%\[ p_t(x,y) \geq C' \frac{1}{V(x,r_t)} \exp\Biggl(-c' \biggl(\frac{\dist(x,y)}{r_t}\biggr)^{\frac{m}{m-1}} \Biggr) \quad (x,y \in \Omega,\: t > 0)
%\]
%for another couple of constants $C',c' > 0$.
\end{defi}

\begin{defi}
Let $(\Omega,\dist,\mu)$ be a space of homogeneous type.
Let $A$ be a self-adjoint operator on $L^2(\Omega)$ generating the semigroup $(T_t)_{t \geq 0}$.
Let $p_0 \in [1,2)$ and $m \in [2,\infty).$
We say that $(T_t)_{t \geq 0}$ satisfies generalised Gaussian estimates (with parameters $p_0,m$) if there exist $c,C < \infty$ such that
\begin{multline}
\label{equ-GGE}
\bigl\| 1_{B(x,r_t)} T_t 1_{B(y,r_t)} \bigr\|_{L^{p_0}(\Omega) \to L^{p_0'}(\Omega)}\\  \leq C |V(x,r_t)|^{-(\frac{1}{p_0} - \frac{1}{p_0'})} \exp \Biggl(-c \biggl(\frac{\dist(x,y)}{r_t} \biggr)^{\frac{m}{m-1}}\Biggr) \quad (x,y \in \Omega,\: t > 0),
\end{multline}
where $r_t = t^{\frac1m}.$
\end{defi}

\begin{remark}
\label{rem-GE-GGE}
According to \cite[Proposition 2.9]{BK02} and \cite[Proposition 2.1]{BK05}, Gaussian estimates \eqref{equ-GE-prelims} with parameter $m \geq 2$ for a semigroup imply generalised Gaussian estimates \eqref{equ-GGE} with parameter $p_0 = 1$ and $m$.
Moreover, according to \cite[Proposition 2.1]{BK05}, generalised Gaussian estimates with parameters $p_0 \in [1,2)$ and $m \geq 2$ imply generalised Gaussian estimates with parameters $p_1 \in [p_0,2)$ and $m$.
Finally, note that \eqref{equ-GGE} implies that $(T_t)_{t \geq 0}$ is a bounded self-adjoint semigroup on $L^2(\Omega)$, so that $A$ is then necessarily positive definite.
\end{remark}

We recall the main theorem on H\"ormander functional calculus on $L^p(Y)$ for Gaussian estimates from \cite{DKK}.

\begin{defi}
Let $p \in (1,\infty),$ $p_Y \in (1,2]$ and $q_Y \in [2,\infty).$
We put 
\begin{equation}
\label{equ-defi-alpha}
\alpha(p,p_Y,q_Y) = \max\biggl(\frac{1}{p},\frac{1}{p_Y},\frac12 \biggr) - \min \biggl(\frac{1}{p},\frac{1}{q_Y},\frac12 \biggr) \in (0,1).
\end{equation}
Informally spoken, this is the length of the segment, which is the convex hull of the points $\frac{1}{p},\frac{1}{p_Y},\frac{1}{q_Y}$ and $\frac12$ sitting on the real line.
\end{defi}

For the next theorem, we recall the notion of $p$-convexity and $q$-concavity of a Banach lattice \cite[Definition 1.d.3]{LTz} and the $p$-convexification of a Banach lattice \cite[p.~215]{RdF}.
Note that with these notions, a space $L^r(\Omega')$ is $r$-convex and $r$-concave and its $s$-convexification is $L^\frac{r}{s}(\Omega')$.

\begin{thm}{\cite[Theorem 4.10]{DKK}}
\label{thm-Hoermander}
Let $(\Omega,\dist,\mu)$ be a space of homogeneous type with a dimension $d$.
Let $A$ be a self-adjoint operator on $L^2(\Omega)$ generating the semigroup $(T_t)_{t \geq 0}$.
Let $p_0 \in [1,2)$ and $m \in [2,\infty)$.
Assume that $(T_t)_{t \geq 0}$ satisfies Gaussian (resp. generalised Gaussian) estimates with parameter $m$ (resp. $p_0,m$).
Let $Y$ be a UMD lattice which is $p_Y$-convex and $q_Y$-concave for some $p_Y \in (1,2]$ (resp. $p_Y \in (p_0,2]$) and $q_Y \in [2,\infty)$ (resp. $q_Y \in [2,p_0')$).
Assume that the convexifications $Y^{p_Y}$ and $(Y')^{q_Y'}$ are also UMD lattices.
If $Y = L^s(\Omega')$ for some $s \in (1,\infty)$, then any $p_Y \in (p_0,s)$ and any $q_Y \in (s,p_0')$ are admissible.
Finally, assume that $A$ has a bounded $\HI(\Sigma_\omega)$ calculus on $L^p(\Omega,Y)$ for some fixed $p \in (1,\infty)$ (resp. $p \in (p_0,p_0')$) and $\omega \in (0,\pi)$.

Then $A$ has a H\"ormander $\Hor^\beta_2$ calculus on $L^p(\Omega,Y)$ with
\[ \beta > \alpha(p,p_Y,q_Y) \cdot d + \frac12\]
and $\alpha$ from \eqref{equ-defi-alpha}.
\end{thm}

\begin{remark}
\label{rem-examples-Gaussian-estimates}
Theorem \ref{thm-Hoermander} applies to a wide range of differential operators in different contexts.
In the following, we list operators having a H\"ormander $\Hor^\beta_2$ calculus on $L^p(\Omega,Y)$ for $1 < p < \infty$, for any UMD lattice $Y$ and $\beta > \alpha(p,p_Y,q_Y) \cdot d + \frac12$.
\begin{enumerate}
\item The heat semigroup on a complete Riemannian manifold with non-negative Ricci curvature \cite{LY}, \cite[p. 3/70 (1.3)]{GriTel}, \cite{Sal}, \cite[Theorem 4.2.1 \& p.~45]{Fen}, \cite[Proposition 4.8]{DKK}.
\item Schr\"odinger operators on connected and complete Riemannian manifolds with non-negative Ricci curvature and locally integrable, positive potential \cite[Section 7.4, (7.8)]{DuOS}, \cite[Proposition 4.8]{DKK}.
\item Other Schr\"odinger and elliptic differential operators acting on $L^2(\Omega)$, where $\Omega \subseteq \R^d$ is an open subset of homogeneous type \cite{Ouh06}, \cite[Section 6.4, in particular Theorems 6.10, 6.11]{Ouh}.

\item Sub-laplacians on Lie groups with polynomial volume growth \cite[Theorem 4.2, Example 2]{Sal}, \cite{Gri}, \cite[Corollary 4.9]{DKK}.
\item Heat semigroups on fractals \cite[(1.4)]{GriTel}, \cite[Section 7.11]{DuOS}.
\item
For a discussion of many further examples where Gaussian estimates as in \eqref{equ-GE-prelims} are satisfied, we refer to \cite[Section 7]{DuOS}, see also \cite[Subsection 5.1]{DKK}.
\end{enumerate}
\end{remark}

\begin{remark}
\label{rem-examples-generalised-Gaussian-estimates}
In the recent past, several operators with generalised Gaussian estimates \eqref{equ-GGE} for some $p_0 > 1$ have been studied.
In these cases we will obtain according to \cite{KuUl}, \cite[Theorem 4.7]{DKK} and Theorem \ref{thm-Hoermander} that $A$ has a $\Hor^\beta_2$ calculus on $L^p(\Omega,L^s(\Omega'))$ for $p_0 < p,s < p_0'$ and 
\begin{equation}
\label{equ-beta-GGE}
\beta > \Biggl( \max\biggl(\frac1p , \frac1s , \frac12\biggr) - \min \biggl(\frac1p, \frac1s, \frac12\biggr) \Biggr) \cdot d + \frac12 .
\end{equation}
Such examples can be found e.g. in \cite[Section 3]{KuUl}, \cite[Section 2]{Bl}, see also \cite[Subsection 5.3]{DKK}.
\end{remark}

\paragraph{A non-self-adjoint example}

Up to now, our examples of $0$-sectorial operators are mostly of the form $A = A_0 \ot \Id_Y$, where $A_0 : L^2(\Omega) \to L^2(\Omega)$ is self-adjoint.
However self-adjointness is not necessary for an operator to have a H\"ormander calculus on $L^2(\Omega)$, as the following natural example shows.

Let $(T_t)_{t \geq 0}$ be a self-adjoint semigroup of contractions on $L^2(\Omega,\mu)$, where $\Omega$ is a space of homogeneous type, with dimension $d \in \N$.
Assume that $(T_t)_{t \geq 0}$ satisfies Gaussian estimates as in \eqref{equ-GE-prelims}.
Then according to \cite[Theorem 3.2]{DSY}, there exists some $r_0 \in [1,2)$ such that for any $p \in (r_0,\infty)$, any $w \in A_{\frac{p}{r_0}}$ (Muckenhoupt weight class \cite[p.~1109]{DSY}) and $s > \frac{d+1}{2}$, the generator $A$ of $(T_t)_{t \geq 0}$ has a H\"ormander calculus on $L^p(\Omega,w d \mu)$.
Picking $p = 2$, the semigroup $(T_t)_{t \geq 0}$ is thus uniformly bounded on $L^2(\Omega,w d\mu)$, but if $w \in A_{\frac{2}{r_0}}$ is not constant, the semigroup $(T_t)_{t \geq 0}$ and the generator $A$ are not self-adjoint on this weighted $L^2$ space.
Indeed, if $p_t(x,y)$ denotes the integral kernel of $(T_t)_{t \geq 0}$ with respect to measure $\mu$, then for $f \in L^\infty(\Omega,w d\mu)$ of compact support,
\begin{align*}
T_t f(x) & = \int_\Omega p_t(x,y) f(y) d \mu(y) \\
& = \int_\Omega \frac{1}{w(y)} p_t(x,y) f(y) w(y) d \mu(y).
\end{align*}
Therefore, the kernel of $T_t$ with respect to weighted measure $w d\mu$ is $q_t(x,y) = \frac{1}{w(y)} p_t(x,y)$, which is not symmetric if e.g. $p_t(x,y) \neq 0$ for all $x,y \in \Omega$ and $w$ is not constant.
Thus, this semigroup is not self-adjoint on $L^2(\Omega,w d\mu)$ yet its generator is $0$-sectorial and has a H\"ormander calculus on this space.

\subsection{Examples in UMD lattices}
\label{subsec-examples-UMD-lattices}

Apart from $L^q(\Omega')$ spaces for $1 < q < \infty$ which are the major examples of UMD lattices, one can also consider their intersections and sums when seen as subspaces of the common superspace of (equivalence classes of) measurable functions over $\Omega'$.

\begin{lemma}
\label{lem-intersection-sum-of-Lq}
Let $1 < q_1,q_2 < \infty$.
Let $\Omega'$ be a $\sigma$-finite measure space.
\begin{enumerate}
\item Then $Y = L^{q_1}(\Omega') \cap L^{q_2}(\Omega')$ equipped with the norm $\norm{f}_Y = \max\{\norm{f}_{q_1},\norm{f}_{q_2}\}$ and obvious partial order is a UMD lattice.
\item Then $Y = L^{q_1}(\Omega') + L^{q_2}(\Omega')$ equipped with the norm $\norm{f}_Y = \inf\{ \norm{g}_{q_1} + \norm{h}_{q_2} : \: f = g + h \}$ and obvious partial order is also a UMD lattice.
\end{enumerate}
\end{lemma}

\begin{proof}
1. It is immediate from the definition of a Banach lattice \cite[Definition 1.a.1]{LTz} that $Y$ is a Banach lattice.
Moreover, it is easy to check that $Y$ satisfies the axioms of a K\"othe function space from Subsection \ref{subsec-UMD-lattices}.
Finally, it is straightforward to check that the Hilbert transform satisfies $\norm{H}_{L^2(\R,Y)} \leq \sqrt{2} \max\{ \norm{H}_{L^2(\R,L^{q_1})} , \norm{H}_{L^2(\R,L^{q_2})} \}$, so that $Y$ is also a UMD space.

2. It follows from \cite[2.7.1 Theorem]{BeL} that $L^{q_1}(\Omega') + L^{q_2}(\Omega')$ is the dual space of $L^{q_1^*}(\Omega') \cap L^{q_2^*}(\Omega)$.
Then it follows from \cite[Theorem B.1.12]{Lin} that $Y= L^{q_1}(\Omega') + L^{q_2}(\Omega')$ is a Banach lattice with the $\sigma$-Fatou property.
Now it is easy to check that $Y$ is a K\"othe function space.
Moreover, since the dual space of a UMD space is again UMD \cite[Proposition 4.2.17]{HvNVW}, the lemma is proved.
\end{proof}

In the remainder of this subsection, we give two natural examples of sectorial operators, the first one defined on a subspace of the UMD lattices from Lemma \ref{lem-intersection-sum-of-Lq}, the second one  defined on a vector valued amplification of them, see \eqref{equ-Amanns-lattice} below.

\paragraph{Analysis of the Stokes operator}
The following operator provides an example where the UMD lattices from Lemma \ref{lem-intersection-sum-of-Lq} appear naturally.
Indeed, we shall see that the Stokes operator on domains $\Omega \subseteq \R^d$ with suitable regularity has a functional calculus on a subspace of $L^q \cap L^2$ resp. $L^q + L^2$, but not on $L^q$ itself.
We recall that the Stokes system is given by
\[ \begin{cases} - \Delta u + \nabla p & = f \text{ in }\Omega, \\
\text{div}(u) & = 0 \text{ in }\Omega, \\
u & = 0 \text{ on }\partial\Omega,
\end{cases} \]
where $f$ is a given function, and $(u,p)$ is the solution \cite[(1.1)]{She}.
Let $L^q_\sigma(\Omega)$ be the closed subspace of $L^q(\Omega,\C^d)$ generated by $\{ g \in C^\infty_c(\Omega,\C^d) : \: \text{div} (g) = 0 \}$.
Associated with the above system is the Stokes operator 
\[ A_q : D(A_q) \subseteq L^q_\sigma(\Omega) \to L^q_\sigma(\Omega),\: u \mapsto - \Delta u + \nabla p, \]
where $p \in L^q(\Omega)$ is the unique function such that $-\Delta u + \nabla p$ belongs to $L^q_\sigma(\Omega)$ \cite[(1.7), (1.8)]{She}, \cite[p.~404]{KW17}.
Then it is well-known that $A_2$ is self-adjoint positive for any non-empty open $\Omega \subseteq \R^d$ \cite[Remark 5.2.2]{Tol}.
However, when $q \neq 2$, the operator $A_q$ has less good properties.
Namely, for $1 < q \neq 2 < \infty$ there exist certain unbounded smooth domains $\Omega \subseteq \R^d$ such that $A_q$ does not generate a semigroup any more \cite[abstract]{FKZ}. 
This is due to the fact that the Helmholtz projection $P:  L^q(\Omega,\C^d) \to L^q_\sigma(\Omega)$, where $Pu = u - \nabla p$ for appropriate $p \in L^q(\Omega)$ as above, is not bounded \cite{FKZ,Ku08}.

A convenient remedy for this obstruction is to consider the spaces from Lemma \ref{lem-intersection-sum-of-Lq}, that is, one puts
\[ \tilde{L}^{q}(\Omega) = \begin{cases} L^q(\Omega,\C^d) \cap L^2(\Omega,\C^d) & : \: 2 \leq q < \infty \\ L^q(\Omega,\C^d) + L^2(\Omega,\C^d) & : \: 1 < q < 2 \end{cases} \]
and also
\[ \tilde{L}^{q}_\sigma(\Omega) = \begin{cases} L^q_\sigma(\Omega) \cap L^2_\sigma(\Omega) & : \: 2 \leq q < \infty \\ L^q_\sigma(\Omega) + L^2_\sigma(\Omega) & : \: 1 < q < 2 \end{cases} \] 
\cite[p.~258]{FKZ}, \cite[p.~178]{Ku08}.
Then $\tilde{L}^q_\sigma(\Omega)$ is an isometric subspace of $\tilde{L}^q(\Omega)$ (since it is complemented by the Helmholtz projection \cite[Theorem 2.1]{FKZ05}).
It is proved in \cite[Theorem 1.3]{FKZ07} that the appropriate version $\tilde{A}_q$ on $\tilde{L}^q_\sigma(\Omega)$ is $0$-sectorial provided $\Omega$ is a uniform $C^{1,1}$ domain, for any $1 < q < \infty$.
Moreover, in \cite[Theorem 1.4]{FKZ}, it is proved that $\tilde{A}_q$ has maximal regularity for the same type of domain $\Omega$ and $q$.
Finally, in \cite[Theorem 1.1]{Ku08} (see \cite[Remark 1.2]{GeKu} for a correction of a gap in the original argument), it is shown that $\tilde{A}_q$ has an $\HI$ calculus provided that $\Omega$ has a uniform $C^{2 + \epsi}$ boundary and $1 < q < \infty$.

For those domains $\Omega$ such that $\tilde{A}_q$ has in addition a H\"ormander calculus on the UMD space $Z: = \tilde{L}^q_\sigma(\Omega) \subseteq Y: = \tilde{L}^q(\Omega)$, our Theorem \ref{thm-main} applies to this operator and yields $\norm{t \mapsto m(tA) f}_{Y(\Lambda^\beta)} \lesssim \norm{m}_{W^c_2(\frac12,2)} \norm{f}_Z$ for appropriate $c$ (the $L^p(\Omega)$ component in this theorem is then void, i.e. consider a one-point space).
Indeed, one can check that Theorem \ref{thm-main} holds in this form for $0$-sectorial operators acting on subspaces of UMD lattices, by the same proof.
Note that the H\"ormander calculus implies the $\HI(\Sigma_\sigma)$ calculus for any angle $\sigma \in (0,\pi)$, the maximal regularity as well as the $0$-sectoriality.

We remark that it is an open question for which domains $\Omega$, $\tilde{A}_q$ has a H\"ormander calculus on $\tilde{L}^q_\sigma(\Omega)$, or even $A_q$ on ${L}^q_\sigma(\Omega)$.
As a step in this direction, in the case that $\Omega$ is a \textit{bounded} Lipschitz domain in dimension $d \geq 3$, \cite[Proposition 13]{KW17} shows that $A_q$ on the original space $L^q_\sigma(\Omega)$ is $R$-sectorial of angle $0$ and thus has an $\HI(\Sigma_\theta)$ calculus for any $\theta \in (0,\pi)$ \cite[Theorem 16]{KW17}, provided $\left|\frac1q - \frac12\right| < \frac{1}{2d} + \epsi$.
Such $\HI$ calculi for arbitrarily small angles are linked to H\"ormander calculus by \cite[Theorem 4.10]{CDMY}.
As a precursor, see also \cite[Theorem 1.1]{She} which proves that $A_q$ is $0$-sectorial on $L^q_\sigma(\Omega)$ for the same $\Omega$ and $q$.
See also \cite{Tol} for related results, and \cite[Theorem 1.1]{GeKu} on more general domains under the assumption that the above Helmholtz projection $P$ is bounded on $L^q(\Omega,\C^d)$.

\paragraph{Coagulation-fragmentation equations}

Examples of differential operators acting on Bochner spaces (or even more general vector valued spaces, such as spaces of bounded uniformly continuous functions and Besov spaces) are considered in \cite{Ama}.
In \cite[(0.1)]{Ama}, the author defines the elliptic differential operator
\begin{equation}
\label{equ-Amanns-operator}
A = A(x,D) = \sum_{|\alpha| \leq m} a_\alpha(x) D^\alpha
\end{equation}
acting on functions $f : \R^n \to E$, with operator valued coefficients $a_\alpha \in B(E)$, where $E$ is an arbitrary Banach space.
Then he denotes the principal part of $A$ by \cite[p.~155]{Ama}
\[ \sigma A : \: \R^n \times \R^n \to B(E), \: (x,\xi) \mapsto \sum_{|\alpha| = m} a_\alpha(x) \xi^\alpha,\]
and assumes that it is uniformly $(\kappa,\vartheta)$-elliptic.
Moreover, the coefficients $a_\alpha$ are subject to H\"older continuity of some order $\rho \in (0,1)$.
Under these assumptions, \cite[Theorem 5.10]{Ama} shows that given any Banach space $E$ and any $1 \leq p < \infty$, the $L^p(\R^n,E)$ realisation of $-A$ is $(\pi - \vartheta)$-sectorial  after some shift, $\omega - A_p$ for a certain $\omega > 0$, under the conditions that $A$ is $(\kappa,\vartheta)$-elliptic in the above sense for some $\kappa > 0$ and $\vartheta \in (0,\pi)$, and $\sum_{|\alpha| \leq m} \norm{a_{\alpha}}_{\rho,\infty} < \infty$ for some $\rho \in (0,1)$.

It is an open question for which spaces $E$ and exponents $1 \leq p < \infty$ the operator $\omega - A_p$ has moreover a (H\"ormander) functional calculus on $L^p(\R^n,E)$.
Note that one has probably to assume Banach space geometric properties on $E$ and take $1 < p < \infty$, in addition to the hypotheses of \cite{Ama}.

The physical motivation for the operator $A$ from \eqref{equ-Amanns-operator} comes from reaction-diffusion equations \cite[(0.4)]{Ama}
\begin{equation}
\label{equ-Amanns-equation}
\partial_t u - \nabla_x \cdot(a(x,y,u) \nabla u)= f(x,y,u) , \quad (x \in \R^n,\: t > 0) 
\end{equation}
where the solution $u = u(t,x,y)$ depends (a part from time $t> 0$ and space variable $x \in \R^n$ also) on a variable $y \in \Omega'$.
Here, $\Omega' = \N$ in case of discrete models and $\Omega' = \R_+$ in the continuous case.
In order to obtain information on the solution of \eqref{equ-Amanns-equation}, one is naturally led to consider a function space like $L^p(\R^n,Y(\Omega'))$, $\R^n$ accomodating the variable $x$ in the physical problem and $\Omega'$ accomodating the variable $y$.
In the simple case of $a$ being independent of $u$, $A$ becomes 
\[ A(x,D) = a(x,y) \sum_{i=1}^n D_{x_i}^2 + \sum_{i=1}^n D_{x_i} a(x,y) D_{x_i}, \]
$a(x,y)$ (and $D_{x_i} a(x,y)$) being interpreted as pointwise multipliers, that is, $a(x,y) : Y(\Omega') \to Y(\Omega'), \: u \mapsto (y \mapsto a(x,y)u(y))$.
Moreover, according to \cite[(0.7)]{Ama}, a natural function space on which the sectorial operator $\omega - A$ will act, and thus for the solution of \eqref{equ-Amanns-equation}, is the lattice
\begin{equation}
\label{equ-Amanns-lattice}
X = (L^q \cap L^p)(\R^d,L^r(\Omega',d \mu'))
\end{equation}
with $q = 1$, $p > n$ and $r = 1$.
% For appropriate $a$ independent of $u$ and satisfying the ellipticity and H\"older regularity assumptions of \cite[Theorem 5.10]{Ama}, \eqref{equ-Amanns-equation} becomes $\partial_t u - Au = f$.
%According to \cite[p.~145]{Ama} and \cite[p.~106]{Paz}, in case of $u$-independent functions $a$ and $f$, \eqref{equ-Amanns-equation} then has a mild solution for appropriate $f$ (Pazy) and $a$ (Amann Theorem 5.10).
Note that for modified exponents $1 < q,p,r < \infty$, $X$ is turned into a UMD lattice (without the K\"othe property 2. from Subsection \ref{subsec-UMD-lattices} in general unless $q = p = r$), which can be shown in a similar way to Lemma \ref{lem-intersection-sum-of-Lq}.
So it becomes a candidate of function space for our Theorem \ref{thm-main}, provided $\omega - A$ has a H\"ormander calculus on $X$.
Again note that sectoriality in vector-valued $L^p$ spaces proved in \cite[Theorem 5.10]{Ama} is a necessary condition for this H\"ormander calculus.
Examples of operators $A$ having a H\"ormander calculus on $X$ as in \eqref{equ-Amanns-lattice} with $1 < r < \infty$ and $p_0 < p,q < p_0'$ for certain $p_0 \in [1,2)$ are
\begin{itemize}
\item Elliptic divergence form operators $Af = \sum_{|\gamma|,|\delta| = m} (-1)^{|\delta|} \partial^\delta(a_{\gamma \delta} \partial^\gamma f)$ with scalar valued $a_{\gamma \delta} \in L^\infty(\R^d;\R)$ and ellipticity condition as in \cite[Subsection 5.3]{DKK} (in particular the Laplace operator $A = - \Delta \ot \Id_{L^r(\Omega')}$),
\item Schr\"odinger operators with singular potentials $A = - \Delta + V$ and suitable potentials $V : \R^d \to \R$ again as in \cite[Subsection 5.3]{DKK}.
\end{itemize}
Indeed, it is proved in \cite{DKK} that these operators have a H\"ormander calculus on $L^p(\R^d,L^r(\Omega'))$ for the above parameters $r$ and $p$.
Then $\norm{m(A)f}_X = \max\{ \norm{m(A)f}_{L^p(L^r)} , \norm{m(A) f}_{L^q(L^r)} \} \leq C \norm{m}_{\Hor^\beta_2} \max\{ \norm{f}_{L^p(L^r)}, \norm{f}_{L^q(L^r)} \} = C \norm{m}_{\Hor^\beta_2} \norm{f}_X$, so that $A$ also has a H\"ormander calculus on $X$.

\subsection{Necessity of type and cotype assumptions}
\label{subsec-type-cotype-necessary}

In Theorem \ref{thm-main} and its consequences, one remarks that there is a gap between the order $\alpha$ of the H\"ormander calculus of $A$ in the hypothesis and the order $c$ in the conclusion of the vectorial or maximal estimate.
Concerning the quantity $\frac{1}{\type L^p(Y)} - \frac{1}{\cotype L^p(Y)}$ which appears in this gap, note that it is only needed to pass from the $\Hor^\alpha_2$ calculus of $A$ to an $R$-bounded $\Hor^\gamma_2$ calculus of $A$ through Proposition \ref{prop-Hormander-calculus-to-R-Hormander-calculus}.
An indication that the type and cotype of $Y$ (which are equal to the type and cotype of $L^2(\Omega,Y)$) are necessary in order to obtain $R$-bounded functional calculus is given in equation \eqref{equ-1-prop-type-cotype-necessary} of Proposition \ref{prop-type-cotype-necessary} below, already in the case of spectral multipliers of $-\Delta$ and Fourier multipliers on euclidean space.
First we have the following technical lemma.

\begin{lemma}
\label{lem-type-cotype-necessary}
Let $(g_n)_n$ be a normalized sequence in $L^2(\R^d)$ such that $\supp (g_n) \cap \supp(g_k) = \emptyset$ for $n \neq k$.
Let $Y$ be a Banach space.
Assume that $(T_n)_n$ is a sequence in $B(L^2(\R^d))$ such that for any $N \in \N$, $\{T_1 \ot \Id_Y,\ldots,T_N \ot \Id_Y\}$ is $R$-bounded over $L^2(\R^d,Y)$ with a control of $R$-bound $\lesssim N^\gamma$ for some $\gamma \geq 0$.
Finally, assume that $T_n g_n = f$ for all $n \in\N$ with $0 \neq f \in L^2(\R^d)$.
\begin{enumerate}
\item If $\gamma = 0$, then $Y$ is of type $2$.
\item If $\gamma > 0$, then $Y$ is of any type $r$ with $\frac{1}{r} > \gamma + \frac12$.
\end{enumerate}
\end{lemma}

\begin{proof}
We let $N \in \N$ and $y_1,\ldots,y_n \in Y$ arbitrary.
Then
\begin{align*}
\MoveEqLeft
\left( \E \norm{\sum_{j=1}^N \epsi_j y_j}_Y^2 \right)^{\frac12} \cong \left( \E \int_{\R^d} |f(x)|^2 dx \norm{ \sum_j \epsi_j y_j}_Y^2 \right)^{\frac12} \\
& = \left( \E \int_{\R^d} \norm{ \sum_j \epsi_j f(x) y_j}_Y^2 dx \right)^{\frac12} \\
& = \left( \E \int_{\R^d} \norm{ \sum_j \epsi_j T_j g_j (x) y_j}_Y^2 dx \right)^{\frac12} \\
& \lesssim N^\gamma \left( \E \int_{\R^d} \norm{ \sum_j \epsi_j g_j(x) y_j}_Y^2 dx\right)^{\frac12} \\
& = N^\gamma \left( \E \sum_{k=1}^N \int_{\supp (g_k)} \norm{ \sum_j \epsi_j g_j(x) y_j}_Y^2 dx \right)^{\frac12} \\
& = N^\gamma \left( \E \sum_{k=1}^N \int_{\supp (g_k)} |g_k(x)|^2 \norm{y_k}_Y^2 dx \right)^{\frac12} \\
& = N^\gamma \left( \sum_{k=1}^N \norm{y_j}_Y^2 \right)^{\frac12}.
\end{align*}
Here, we have used the $R$-boundedness of $\{T_1 \ot \Id_Y,\ldots,T_N \ot \Id_Y\}$ in the fourth line, the disjointness of the supports of the $g_j$ in the penultimate line, and the $L^2$-normalisation of the $g_j$ in the last line.
If $\gamma = 0$, we immediately deduce that $Y$ is of type $2$.
If $\gamma > 0$, then we deduce that $Y$ is of weak Rademacher type $r$ with $\frac1r = \gamma + \frac12$, which implies that $Y$ is actually of usual Rademacher type $r$ with $\frac1r > \gamma + \frac12$ \cite{Mas}.
\end{proof}

In the following proposition, we write $B_\infty(\R^d)$ for the Banach algebra of bounded Borel functions over $\R^d$ (similarly $B_\infty(\R_+)$), and for $\phi \in B_\infty(\R^d)$, $M_\phi$ the Fourier multiplier with symbol $\phi$.
Moreover, for $\beta > \frac{d}{q}$ we write $\Hor^\beta_q(\R^d) = \{ m \in C_b(\R^d) :\: \norm{m}_{\Hor^\beta_q(\R^d)} = \sup_{t > 0} \norm{\phi m(t \cdot)}_{W^{\beta}_q(\R^d)} < \infty \}$ the H\"ormander symbols over $d$-dimensional space.
Hereby, $\phi \in C^\infty_c(\R^d)$ is a non-zero radial function with support in $B(0,4)\backslash B(0,1)$.

\begin{prop}
\label{prop-type-cotype-necessary} 
Let $Y$ be a Banach space and $d \in \N$ a dimension.
\begin{enumerate}
\item Assume that $-i \frac{d}{dx} \ot \Id_Y$ has an $R$-bounded $B_\infty(\R)$ calculus over $L^2(\R,Y)$.
Then $Y$ is isomorphic to a Hilbert space.
The same conclusion holds if $-\Delta \ot \Id_Y$ has an $R$-bounded $B_\infty(\R_+)$ calculus over $L^2(\R,Y)$ or if the family $\{ M_\phi \ot \Id_Y : \: \phi \in B_\infty(\R^d),\: \norm{\phi}_{B_\infty(\R^d)} \leq 1 \}$ is $R$-bounded over $L^2(\R^d,Y)$.
\item Assume that $-i \frac{d}{dx} \ot \Id_Y$ has an $R$-bounded $\Hor^\alpha_s(\R^1)$ calculus over $L^2(\R,Y)$ for some $1 < s < \infty$ and $\alpha > \frac1s$.
Then $Y$ is of type $p$ and cotype $q$ where $\frac1p, 1 - \frac1q > \alpha + \frac12$.
E.g. if $Y$ is an $L^r$ space, then necessarily 
\begin{equation}
\label{equ-1-prop-type-cotype-necessary}
\alpha \geq | \frac1r - \frac12 | = \frac{1}{\type L^r} - \frac{1}{\cotype L^r} .
\end{equation}
The same holds if $-\Delta \ot \Id_Y$ has an $R$-bounded $\Hor^\alpha_s(\R_+)$ calculus over $L^2(\R,Y)$ or if the family $\{ M_\phi \ot \Id_Y :\: \phi \in \Hor^{\alpha \cdot d}_s(\R^d) , \: \norm{\phi}_{\Hor^{\alpha \cdot d}_s(\R^d)} \leq 1 \}$ is $R$-bounded over $L^2(\R^d,Y)$.
\item Conversely, if $Y$ is an $L^r$ space with $1 < r < \infty$, then for any $\alpha >  | \frac1r - \frac12 |$, if $d$ is sufficiently large, then $-\Delta \ot \Id_Y$ has an $R$-bounded $\Hor^{\alpha \cdot d}_2(\R_+)$ calculus over $L^2(\R^d,Y)$.
\end{enumerate}
\end{prop}

\begin{proof}
1. We apply Lemma \ref{lem-type-cotype-necessary} with $g_j = 1_{[j,j+1)}$ and $T_j = M_{e^{ij(\cdot)}}$ the shift operator mapping $g_j$ to $f = g_0$.
Since $T_j = \exp(ij (-i \frac{d}{dx}))$ is a $B_\infty(\R)$ spectral multiplier of $-i \frac{d}{dx}$, we infer by the assumptions that $\{ T_j \ot \Id:\:j \in \N\}$ is $R$-bounded over $L^2(\R,Y)$.
Thus, Lemma \ref{lem-type-cotype-necessary} applies with $\gamma = 0$, and $Y$ is of type $2$.
Moreover, since $m(-i \frac{d}{dx})^* = \ovl{m}(-i \frac{d}{dx})$ with $\norm{m}_{B_\infty(\R)} = \norm{\ovl{m}}_{B_\infty(\R)}$ and $R$-boundedness is stable under passage to adjoints, the assumptions imply that $-i \frac{d}{dx} \ot \Id_{Y^*}$ also has an $R$-bounded $B_\infty(\R)$ calculus, over $L^2(\R,Y^*)$. Thus, we deduce that $Y^*$ is also of type $2$, so by \cite[11.10 Proposition p.220]{DiJT}, $Y$ is of cotype $2$.
Finally, by a well-known result of Kwapien, any space of type $2$ and cotype $2$ is isomorphic to a Hilbert space.

If $-\Delta \ot \Id_Y$ has an $R$-bounded $B_\infty(\R_+)$ calculus, then by Guerre-Delabri\`ere's theorem \cite[Corollary 10.5.2 p.441]{HvNVW2}, $Y$ is a UMD space.
Moreover, since $m \mapsto m((\cdot)^2)$ is an isometry in $B_\infty(\R_+)$, also $(-\Delta)^{\frac12} \ot \Id_Y$ has an $R$-bounded $B_\infty(\R_+)$ calculus.
For any $m \in B_\infty(\R)$, write $m_+(\xi) = m(\xi) 1_{[0,\infty)}(\xi)$ and $m_-(\xi) = m(-\xi) 1_{(0,\infty)}(\xi)$.
Then 
$\norm{m_+}_{B_\infty(\R_+)},\norm{m_-}_{B_\infty(\R_+)} \leq \norm{m}_{B_\infty(\R)}$ and 
\[m(-i\frac{d}{dx}) = m_+((-\Delta)^{\frac12})H + m_-((-\Delta)^{\frac12})(\Id - H),\]
where $H = M_{1_{[0,\infty)}}$ is the Hilbert transform.
Since $Y$ is UMD, $H \ot \Id_Y$ and $(\Id - H) \ot \Id_Y$ are bounded on $L^2(\R,Y)$.
We infer that $-i \frac{d}{dx} \ot \Id_Y$ then has an $R$-bounded $B_\infty(\R)$ calculus over $L^2(\R,Y)$, and can apply the first part of the proof.

If the family $\{ M_\phi \ot \Id_Y : \: \phi \in B_\infty(\R^d),\: \norm{\phi}_{B_\infty(\R^d)} \leq 1 \}$ is $R$-bounded over $L^2(\R^d,Y)$, we argue similarly with Lemma \ref{lem-type-cotype-necessary} using $g_j(x) = c \cdot 1_{B(0,\frac12)}(x-j)$ and $T_j g(x) = g(x+j)$ with $j \in \N^d$.

2. We again apply Lemma \ref{lem-type-cotype-necessary} with $g_j = 1_{[j,j+1)}$, $j \in \N$, but this time, we choose $T_j = m_j(-i \frac{d}{dx})$ with $m_j(\xi) = \frac{1}{(1+ \xi^2)^{\alpha/2}} e^{ij \xi}$.
It is essentially shown in \cite[Lemma 3.9]{KrW3} that $\norm{m_j}_{\Hor^\alpha_s(\R^1)} \lesssim (1 + j)^\alpha$.
We then have $T_jg_j = M_\phi g_0 =: f \neq 0$ with $\phi(\xi) = \frac{1}{(1 + \xi^2)^{\alpha/2}}$.
Since the assumptions on $R$-bounded calculus imply that $\{T_1 \ot \Id_Y,\ldots,T_N \ot \Id_Y\}$ is $R$-bounded with control $\lesssim N^\alpha$, Lemma \ref{lem-type-cotype-necessary} implies that $Y$ is of type $p$ with $\frac{1}{p} > \alpha + \frac12$.
For the cotype statement, similarly to part 1., we can employ a duality argument.
Indeed, $\norm{m}_{\Hor^\alpha_s(\R^1)} = \norm{\ovl{m}}_{\Hor^\alpha_s(\R^1)}$ and $(m(-i \frac{d}{dx}) \ot \Id_Y)^* = \ovl{m}(-i \frac{d}{dx}) \ot \Id_{Y^*}$.
Thus, $-i \frac{d}{dx} \ot \Id_{Y^*}$ has an $R$-bounded $\Hor^\alpha_s(\R^1)$ calculus over $L^2(\R,Y^*)$.
We infer by the above that $Y^*$ has type $q^*$ with $\frac{1}{q^*} > \alpha + \frac12$, so that by \cite[11.10 Proposition]{DiJT}, $Y$ has cotype $q$ with $1 - \frac{1}{q} > \alpha + \frac12$.
Since an $L^r$ space satisfies $\max\left(\frac{1}{\type L^r},1 - \frac{1}{\cotype L^r}\right) = |\frac1r - \frac12| + \frac12$ (with no better type and cotype), we infer that necessarily, $\alpha$ satisfies \eqref{equ-1-prop-type-cotype-necessary}.

If $-\Delta \ot \Id_Y$ has an $R$-bounded $\Hor^\alpha_s(\R_+)$ calculus, then similarly as in 1., we deduce that $-i \frac{d}{dx} \ot \Id_Y$ has an $R$-bounded $\Hor^\alpha_s(\R^1)$ calculus.
Indeed, again by \cite[Corollary 10.5.2 p.441]{HvNVW2}, $Y$ must be a UMD space so that $H \ot \Id_Y = M_{1_{[0,\infty)}} \ot \Id_Y$ is bounded on $L^2(\R,Y)$.
Then for $m \in \Hor^\alpha_s(\R^1)$, we write
$m(-i \frac{d}{dx}) = m_+((-\Delta)^{\frac12}) H + m_-((-\Delta)^{\frac12}) (\Id - H)$.
Since $m \mapsto m((\cdot)^2)$ is an isomorphism on $\Hor^\alpha_s(\R_+)$, $(-\Delta)^{\frac12} \ot \Id_Y$ has an $R$-bounded $\Hor^\alpha_s(\R_+)$ calculus over $L^2(\R,Y)$.
Then we conclude since $\norm{m_+}_{\Hor^\alpha_s(\R_+)} , \norm{m_-}_{\Hor^\alpha_s(\R_-)} \leq \norm{m}_{\Hor^\alpha_s(\R^1)}$.

Finally, in case that $\{ M_\phi \ot \Id_Y :\: \phi \in \Hor^{\alpha \cdot d}_s(\R^d) , \: \norm{\phi}_{\Hor^{\alpha \cdot d}_s(\R^d)} \leq 1 \}$ is $R$-bounded over $L^2(\R^d,Y)$, we again use Lemma \ref{lem-type-cotype-necessary}, but this time with $g_j(x) = c \cdot 1_{B(0,\frac12)}(x-j)$ indexed by $j \in \N^d$.
Moreover, we take $T_j = M_{m_j}$ with $m_j(\xi) = \frac{1}{(1 + |\xi|^2)^{\alpha \cdot d/2}} e^{i j \cdot \xi}$.
Similarly to the above, one can show that $\norm{m_j}_{\Hor^{\alpha \cdot d}_s(\R^d)} \leq (1 + |j|)^{\alpha \cdot d}$.
We have $T_j g_j = M_\phi g_0 =:f \neq 0$ with $\phi(\xi) = \frac{1}{(1 + |\xi|^2)^{\alpha \cdot d}}$.
Then for $N \in \N$, the family $\{ T_j \ot \Id_Y = T_{(j_1,\ldots,j_d)} \ot \Id_Y :\: j_1,\ldots,j_d \in \{1,\ldots,N\} \}$ is of cardinality $N^d$ and of $R$-bound controlled by $N^{\alpha \cdot d}$.
Thus, Lemma \ref{lem-type-cotype-necessary} applies with $\gamma = \frac{\alpha \cdot d}{d} = \alpha$, and thus yields that $Y$ is of type $r$ with $\frac1r > \alpha + \frac12$ (and similarly for the cotype by a duality argument).

3. Since the heat semigroup satisfies Gaussian estimates and an $L^r$ space is $r$-convex and $r$-concave, \cite[Corollary 4.13, (4.1)]{DKK} applies and yields a $\Hor^{\beta \cdot d + \frac12}_2(\R_+)$ calculus for $-\Delta \ot \Id_{L^r}$ on $L^2(\R^d,L^r)$ with $\beta > \frac{1}{\type L^r} - \frac{1}{\cotype L^r} = |\frac1r - \frac12|$.
If also $\alpha > |\frac1r - \frac12|$, then note $\alpha \cdot d = \beta \cdot d + \frac12 \Longleftrightarrow \alpha = \beta + \frac{1}{2 \cdot d}$, so that we conclude by taking $\beta$ in between $|\frac1r - \frac12|$ and $\alpha$, and $d$ sufficiently large so that $\frac{1}{2\cdot d} = \alpha - \beta$.
\end{proof}

\section{Concluding remarks}

Already for classical (i.e. non-maximal/$q$-variational) H\"ormander multiplier theorems, a nice description of the exact norm $\|m(A)\|_{L^p \to L^p}$ in terms of a function norm $\|m\|$ of the spectral multiplier is not known today.
This problem is equally present for our maximal spectral multipliers in this article and only a step by step progression of sufficient conditions in the form $\|t \mapsto m(tA)(\cdot)\|_{L^p \to L^p(L^\infty(\R_+))} \lesssim \|m\|_{\ldots}$ seems to be manageable.
In this direction, it would be interesting to know whether in the context of Corollary \ref{cor-main-full-support} of semigroup generators, one can relax the summation condition to
\[ \sum_{n \in \Z} \frac{\|m(2^n\cdot)\dyad_0\|_{W^c_2(\R)}}{1 + |n|} < \infty \]
as in the euclidean case \cite{CGHS,GHS}, or with an additional factor $\log(|n|+2)$ as in \cite{Choi}.
As another possible relaxation, one can ask the question, whether
\[ \sum_{n \in \Z} \|m(2^n \cdot)\dyad_0\|_{W^c_q(\R)}^2 < \infty \]
is sufficient for a maximal estimate, which is known to be true for the euclidean Laplacian \cite[(1.3)]{CGHS}.
Also maximal estimates for spectral multipliers that do not decay at $\infty$ are not well understood.
Already the scalar case $Y = \C$ would be interesting.

As a partial result for radial spectral multipliers of the euclidean Laplacian, and on the radial part $L^p_{\rad}(\R^d)$, see \cite{HNS} for a description of $\|m(-\Delta)\|_{L^p\to L^p}$ and \cite{Kim} for a description of $\|t\mapsto m(-t\Delta)(\cdot)\|_{L^p\to L^p(L^\infty(\R_+))}$ in terms of the associated convolution kernel of $m(-\Delta)$.

\section{Acknowledgments}

The authors acknowledge financial support through the research program ANR-18-CE40-0021 (project HASCON).
They also respectively acknowledge financial support through the research program ANR-18-CE40-0035  (project REPKA) and ANR-17-CE40-0021 (project FRONT).
We thank Peer Kunstmann for explanations concerning the Stokes operator from Subsection \ref{subsec-examples-UMD-lattices}.

\vspace{0.2cm}

\footnotesize{
\noindent Luc Deleaval \\
\noindent
Laboratoire d'Analyse et de  Math\'ematiques Appliqu\'ees (UMR 8050)\\
Universit\'e Gustave Eiffel \\
5, Boulevard Descartes, Champs sur Marne\\
77454 Marne la Vall\'ee Cedex 2\\
luc.deleaval@univ-eiffel.fr\hskip.3cm
}

\vspace{0.2cm}
\footnotesize{
\noindent Christoph Kriegler\\
\noindent
Universit\'e Clermont Auvergne,\\
CNRS,\\
LMBP,\\
F-63000 CLERMONT-FERRAND,\\
FRANCE \\
URL: \href{https://lmbp.uca.fr/~kriegler/indexenglish.html}{https://lmbp.uca.fr/{\raise.17ex\hbox{$\scriptstyle\sim$}}\hspace{-0.1cm} kriegler/indexenglish.html}\\
christoph.kriegler@uca.fr\hskip.3cm
}
\end{document}